\documentclass[11pt,twoside,reqno]{amsart}

\usepackage{microtype}
\usepackage[OT1]{fontenc}
\usepackage{type1cm}
\usepackage{amssymb}
\usepackage{mathrsfs}
\usepackage{ae,aecompl}
\usepackage{enumerate}
\usepackage{pgf,tikz}
\usetikzlibrary{arrows}
\usepackage{subfigure}

\usepackage[left=2.3cm,top=3cm,right=2.3cm]{geometry}

\numberwithin{equation}{section}

\theoremstyle{plain}
\newtheorem{maintheorem}{Theorem}

\newtheorem{theorem}{Theorem}[section]

\newtheorem{proposition}[theorem]{Proposition}
\newtheorem{prop}[theorem]{Proposition}
\newtheorem{lemma}[theorem]{Lemma}

\newtheorem{claim}{Claim}

\theoremstyle{remark}
\newtheorem{remark}[theorem]{Remark}

\newtheorem*{ack}{Acknowledgement}

\theoremstyle{definition}

\newtheorem{question}[theorem]{Question}

\newcommand{\HH}{\mathcal{H}}

\newcommand{\LL}{\mathcal{L}}
\newcommand{\R}{\mathbb{R}}

\newcommand{\N}{\mathbb{N}}

\newcommand{\iii}{\mathbf{i}}
\newcommand{\jjj}{\mathbf{j}}
\newcommand{\kkk}{\mathbf{k}}
\newcommand{\eps}{\varepsilon}
\newcommand{\fii}{\varphi}
\newcommand{\roo}{\varrho}

\newcommand{\la}{\langle}
\newcommand{\ra}{\rangle}

\DeclareMathOperator{\diam}{diam}
\DeclareMathOperator{\proj}{proj}

\DeclareMathOperator{\sgn}{sgn}

\DeclareMathOperator{\linspan}{span}

\DeclareMathOperator{\esssup}{ess\,sup}
\DeclareMathOperator{\essinf}{ess\,inf}

\DeclareMathOperator{\udimloc}{\overline{dim}_{loc}}
\DeclareMathOperator{\ldimloc}{\underline{dim}_{loc}}

\DeclareMathOperator{\dimh}{dim_H}

\DeclareMathOperator{\ldimh}{\underline{dim}_H}

\DeclareMathOperator{\udimm}{\overline{dim}_M}

\DeclareMathOperator{\diml}{dim_L}
\DeclareMathOperator{\dimaff}{dim_{aff}}

\renewcommand{\iint}{\int\hspace{-0.1in}\int}

\renewcommand{\epsilon}{\varepsilon}
\renewcommand{\phi}{\varphi}

\newcommand{\ii}{\mathbf{i}}

\newcommand{\jj}{\mathbf{j}}
\newcommand{\tv}{\mathbf{v}}
\newcommand{\Av}{\mathbf{A}}

\newcommand{\pv}{\mathbf{p}}

\newcommand{\dd}{\,\mathrm{d}}
\providecommand{\inner}[2]{\langle#1,#2\rangle}

\newcommand{\Id}{\textrm{Id}}

\newcommand{\bi}{\mathbf{i}}

\newcommand{\bj}{\mathbf{j}}

\newcommand{\Fb}{\Phi}

\begin{document}

\title{Dimension of self-affine sets for fixed translation vectors}

\author{Bal\'azs B\'ar\'any}
\address[Bal\'azs B\'ar\'any]{
  Mathematics Institute \\
  University of Warwick \\
  Coventry CV4 7AL \\
  United Kingdom \&
  Budapest University of Technology and Economics \\
  MTA-BME Stochastics Research Group \\
  P.O.\ Box 91 \\
  1521 Budapest \\
  Hungary}
\email{balubsheep@gmail.com}

\author{Antti K\"aenm\"aki}
\address[Antti K\"aenm\"aki]{
  Department of Mathematics and Statistics \\
  P.O.\ Box 35 (MaD) \\
  FI-40014 University of Jyv\"askyl\"a \\
  Finland}
\email{antti.kaenmaki@jyu.fi}

\author{Henna Koivusalo}
\address[Henna Koivusalo]{
  Faculty of Mathematics \\
  University of Vienna \\
  Oskar Morgensternplatz 1 \\
  1090 Vienna, Austria }
\email{henna.koivusalo@univie.ac.at}

\thanks{BB and HK were partially supported by Stiftung Aktion \"Osterreich Ungarn (A\"OU) grants 92\"ou6. BB acknowledges the support from the grant OTKA K104745, OTKA K123782, NKFI PD123970 and the J\'anos Bolyai Research Scholarship of the Hungarian Academy of Sciences, and HK from EPSRC EP/L001462 and Osk.\ Huttunen foundation.}
\date{\today}

\begin{abstract}
  An affine iterated function system is a finite collection of affine invertible contractions and the invariant set associated to the mappings is called self-affine. In 1988,
  Falconer proved that, for given matrices, the Hausdorff dimension of the self-affine set is the affinity dimension for Lebesgue almost every translation vectors. Similar statement was proven by Jordan, Pollicott, and Simon in 2007 for the dimension of self-affine measures. In this
  article, we have an orthogonal approach. We introduce a class of self-affine systems in which, given translation vectors, we get the same results for Lebesgue almost all matrices. The proofs rely on Ledrappier-Young theory that was recently verified for affine iterated function systems by B\'ar\'any and K\"aenm\"aki, and a new transversality condition, and in particular they do not depend on properties of the Furstenberg measure. This allows our results to hold for self-affine sets and measures in any Euclidean space.
\end{abstract}

\maketitle

\section{Introduction}

For a non-singular $d\times d$ matrix $A \in GL_d(\R)$ and a translation vector $v\in\R^d$, let us denote the affine map $x\mapsto Ax+ v$ by $f=f(A,v)$.
Let $\Av=(A_1,\dots, A_N) \in GL_d(\R)^N$ be a tuple of contractive non-singular $d\times d$ matrices and let $\tv=(v_1,\dots,v_N)\in(\R^d)^N$ be a tuple
of translation vectors. Here and throughout we assume that $N \ge 2$ is an integer.
The tuple $\Fb_{\Av,\tv}=(f_1, \dots, f_N)$ obtained from the affine mappings $f_i=f(A_i, v_i)$ is called the \emph{affine iterated function system (IFS)}.
Hutchinson \cite{Hutchinson81} showed that for each $\Phi_{\Av,\tv}$ there exists a unique non-empty compact set $E = E_{\Av,\tv}$ such that
\begin{equation*}
  E=\bigcup_{i=1}^Nf_i(E).
\end{equation*}
The set $E_{\Av,\tv}$ associated to an affine IFS $\Phi_{\Av,\tv}$ is called \emph{self-affine}. In the special case where each of the linear maps
$A_i$ is a scalar multiple of an isometry, we call $\Phi_{\Av,\tv}$ a \emph{similitude iterated function system} and the set $E_{\Av,\tv}$ \emph{self-similar}.

The dimension theory of self-similar sets satisfying a sufficient separation condition was completely resolved by Hutchinson \cite{Hutchinson81}.
Without separation, i.e.\ when the images $f_i(E)$ and $f_j(E)$ can have severe overlapping, the problem is more difficult. The most recent progress in this direction is by Hochman \cite{hochman2015self, Hochman14}.
Among other things, he managed to calculate the Hausdorff dimension of a self-similar set on
the real line under very mild assumptions.

In contrast, the dimension theory of self-affine sets and measures is still far from being fully understood. Traditionally, while working on the topic, it has been common to focus on specific subclasses of self-affine sets, for which more methods are available. One such standard subclass is that of self-affine carpets. In this class special relations between the affine maps are imposed, which makes the structure of the self-affine set more tractable. For recent results for self-affine carpets, see \cite{FergusonFraserSahlsten2015,fraser2014dimensions,KaenmakiOjalaRossi2016}. Another method of study and a class of self-affine sets to which it applies was introduced by Falconer \cite{Falconer88} and later extended by Solomyak~\cite{Solomyak98}. They proved that for a fixed matrix tuple $\Av = (A_1,\ldots,A_N) \in GL_d(\R)^N$, with the operator norms $\|A_i\|$ strictly less than $1/2$, the Hausdorff dimension of the self-affine set $E_{\Av,\tv}$,
$\dimh(E_{\Av,\tv})$, is the affinity dimension of $\Av$, $\dimaff(\Av)$, for $\LL^{dN}$-almost all $\tv \in (\R^d)^N$. Here $\LL^d$ is the $d$-dimensional
Lebesgue measure and the affinity dimension, defined below, is a number depending only on $\Av$. A similar result, due to Jordan, Pollicott, and Simon \cite{JordanPollicottSimon07}, also holds for self-affine measures.

Let us next give an intuitive explanation for Falconer's result. It is easy to see that
\begin{equation} \label{eq:limit-set}
  E = \bigcap_{n=1}^\infty \bigcup_{i_1,\ldots,i_n \in \{1,\ldots,N\}} f_{i_1} \circ \cdots \circ f_{i_n}(B(0,R)),
\end{equation}
where $B(0,R)$ is the closed ball centered at the origin with radius $R = \max_{i \in \{1,\ldots,N\}}|v_i|/(1-\max_{i \in \{1,\ldots,N\}}\|A_i\|) > 0$. Since we are interested in the dimension of $E$ we may, by rescaling if necessary, assume that $R=1$. We immediately see from \eqref{eq:limit-set} that for each fixed $n$ the
sets $f_{i_1} \circ \cdots \circ f_{i_n}(B(0,1))$ form a cover for the self-affine set. For any $A \in GL_d(\R)$, let $1>\alpha_1(A)\ge\cdots\ge\alpha_d(A)>0$ be the
lengths of the principal semiaxes
of the ellipse $A(B(0,1))$. Observe that $f_{i_1} \circ \cdots \circ f_{i_n}(B(0,1))$ is a translated copy of $A_{i_1} \cdots A_{i_n}(B(0,1))$. To find the Hausdorff dimension
of $E$, it is necessary to find optimal covers for $E$. Natural candidates for such covers come immediately from \eqref{eq:limit-set}. In $\R^2$, we need approximately
$\alpha_1(A_{i_1}\cdots A_{i_n})/\alpha_2(A_{i_1}\cdots A_{i_n})$ many balls of radius $\alpha_2(A_{i_1}\cdots A_{i_n})$ to cover $f_{i_1} \circ \cdots \circ f_{i_n}(B(0,R))$.
By the definition of the $s$-dimensional Hausdorff measure $\HH^s$, it follows that
\begin{align*}
  \HH^s(E) &\lesssim \lim_{n \to \infty} \sum_{i_1,\ldots,i_n \in \{1,\ldots,N\}} \frac{\alpha_1(A_{i_1} \cdots A_{i_n})}{\alpha_2(A_{i_1} \cdots A_{i_n})} \alpha_2(A_{i_1} \cdots A_{i_n})^{s} \\
  &= \lim_{n \to \infty} \sum_{i_1,\ldots,i_n \in \{1,\ldots,N\}} \alpha_1(A_{i_1} \cdots A_{i_n}) \alpha_2(A_{i_1} \cdots A_{i_n})^{s-1}.
\end{align*}
The \emph{singular value pressure} $P_{\Av}$ of $\Av$ in this case is
\begin{equation*}
  P_{\Av}(s) = \lim_{n \to \infty} \tfrac{1}{n} \log \sum_{i_1,\ldots,i_n \in \{1,\ldots,N\}} \alpha_1(A_{i_1} \cdots A_{i_n}) \alpha_2(A_{i_1} \cdots A_{i_n})^{s-1}.
\end{equation*}
For the complete definition, see \eqref{esap}. The function $s \mapsto P_{\Av}(s)$ is strictly decreasing and it has a unique zero. If $P_{\Av}(s)<0$, then the sum above is
strictly less than one
for all large enough $n$. Therefore, defining $\dimaff(\Av)$ to be the minimum of $2$ and $s$ for which $P_{\Av}(s)=0$, we have $\dimh(E_{\Av,\tv}) \le \dimaff(\Av)$ for
all $\tv \in (\R^2)^N$. The question then becomes, when are the covers obtained in this way optimal. It is easy to find situations in which some other cover is more efficient;
see Figure \ref{fig:bad-cases}. Intuitively, since the role of the translation vector is to determine the placement of the ellipses, Falconer's result asserts that one never
encounters these situations with a random choice of translation vectors.

\begin{figure}[t]
  \includegraphics[scale=0.55,angle=270]{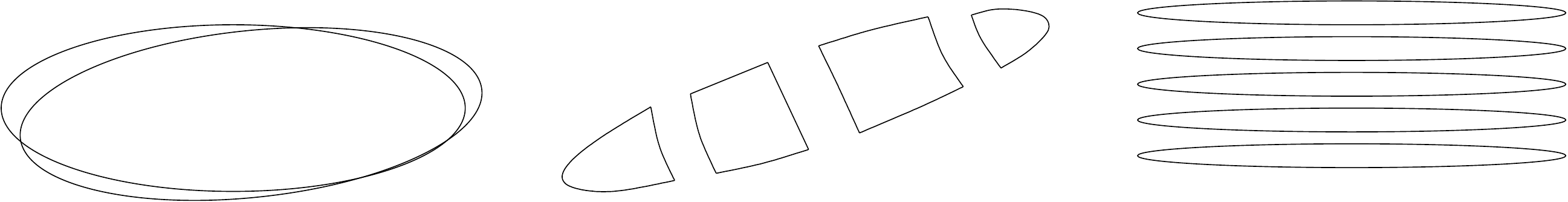}
  \caption{The picture illustrates three example cases where some other covering is more optimal than the one obtained from \eqref{eq:limit-set}. From left to right: ellipses have severe overlapping, the ellipsis does not contain $E$ all the way, and ellipses are badly aligned.} \label{fig:bad-cases}
\end{figure}

Recently, there is an increasing amount of activity in studying the case of general affine iterated function systems, based neither on the strict structure of the self-affine carpets nor Lebesgue generic translation vectors. Morris and Shmerkin \cite{MorrisShmerkin2016} proved that $\dimh(E_{\Av,\tv}) = \dimaff(\Av)$, under both an exponential separation condition on the matrices, that was first introduced by Hochman and Solomyak \cite{HochmanSolomyak2016}, and a separation condition on the IFS. An interesting observation is that the result of Hueter and Lalley \cite{HueterLalley95} can be covered as a special case of \cite[Theorem~1.3]{MorrisShmerkin2016}, and hence, the techniques used by Morris and Shmerkin give it an alternative proof. In particular, if a matrix tuple in $GL_2(\R)^N$ satisfy the dominated splitting condition (for a precise definition, see \S \ref{section:setup}) and the strong separation condition on the projective line, then it satisfies the exponential separation condition. Since both the dominated splitting and strong separation conditions are open properties (i.e.\ if a matrix tuple satisfies it, then it holds in its open neighbourhood), there exists an open set, where the exponential separation condition holds in $GL_2(\R)^N$. In general, the exponential separation condition holds on a dense $G_{\delta\sigma}$ set in $GL_2(\R)^N$, but it is unknown whether it is satisfied by measure theoretically generic tuples of matrices.

Morris and Shmerkin need to further assume that either a so called bunching condition holds, or argue through an application of a result of Rapaport \cite{Rapaport2015}, which assumes that the dimension of the Furstenberg measure (for the definition, see \S \ref{sec:transversality}) is large compared to the Lyapunov dimension. Similarly, Falconer and Kempton \cite{falconer2015planar} prove that $\dimh(E_{\Av,\tv}) = \dimaff(\Av)$, assuming a positivity condition on the matrices, a separation condition on the IFS, and a condition on the dimension of the Furstenberg measure. Both the results of Morris and Shmerkin, and of Falconer and Kempton, rely on calculating the dimension of the Furstenberg measure and can, with the current knowledge of Furstenberg measures, only be applied in the plane.

Our approach combines Ledrappier-Young theory, which was recently proven to hold for many measures on self-affine sets by B\'ar\'any and K\"aenm\"aki \cite{barany2015,BaranyKaenmaki}, and a transversality argument. Our results are a natural counterpart to Falconer's result \cite{Falconer88}: we fix the tuple of translation vectors and investigate the dimension for different choices of matrix tuples. In the same vein as with the intuitive explanation of Falconer's result, one expects that, keeping the centers of the ellipses fixed, a small random change in the shape of
the ellipses guarantees that the covers obtained from \eqref{eq:limit-set} are optimal. Indeed, in the main results of the paper, Theorems \ref{thm:mainplanar} and
\ref{thm:main}, we fix a tuple of distinct translation vectors $\tv \in (\R^d)^N$, and show that $\dimh(E_{\Av,\tv}) = \dimaff(\Av)$ for $\LL^{d^2N}$-almost all $\Av$ in a large
open set of matrix tuples. Notably, a separation condition holds in this open set and for $d\ge 3$ we also need to impose a totally dominated splitting condition (see \eqref{eq:d}) on the matrices. The sharpness and possible extensions are discussed in Remarks \ref{rem:diff-proofs} and \ref{ex:nonsharpness}.

A key ingredient in the proof is a verifiable transversality condition, which we call the modified transversality condition, which we introduce in a general setting in \S \ref{sec:transversality}. This condition allows us to calculate the Hausdorff dimension of self-affine sets and measures through the Ledrappier-Young formula. Therefore, in order to prove the main theorems it suffices to verify the modified transversality condition in the particular setups. We note that the method for calculating dimensions of measures on self-affine sets, as described in \S \ref{sec:transversality}, is rather general, and immediately applies to give stronger results, if there are improvements on the existing results on Ledrappier-Young theory and transversality arguments that the proofs rely on. Another curious feature of our results is that the planar case is different from the higher dimensional case both in statement and in proof; see Remark \ref{rem:diff-proofs} for comparison. The planar case is stated in Theorem \ref{thm:mainplanar} and the higher dimensional case in Theorem \ref{thm:main}.

Since our proofs do not rely on dimension estimates for the Furstenberg measures, our results hold not only for dimensions of  self-affine sets but also self-affine measures, and for any ambient space $\R^d$, not just in the plane. Furthermore, our results hold for an open set of matrix tuples even in parts of the space where the Furstenberg measure has a small dimension compared to the Lyapunov dimension, namely, when the bunching condition does not hold; see Remark \ref{rem:remark2.2}. This is in stark contrast to the earlier works.

The remainder of the paper is organized as follows. In \S \ref{section:setup}, we give a detailed explanation of the setting and state our main results. We explore in \S \ref{sec:transversality} how the Hausdorff dimension of an ergodic measure satisfying the Ledrappier-Young formula can be calculated under a modified self-affine transversality condition. In \S \ref{sec:sub-systems}, we prove an analogous result for self-affine sets. Finally, the proof of Theorem \ref{thm:mainplanar} is given in \S \ref{sec:first-proof} and Theorem \ref{thm:main} is proved in \S \ref{sec:higherdimcase}.

\section{Preliminaries and statements of main results}\label{section:setup}

Let $\Sigma$ be the set of one-sided words of symbols $\left\{1,\dots,N\right\}$ with infinite length, i.e.\ $\Sigma=\left\{1,\dots,N\right\}^{\N}$. Let us denote the left-shift operator on $\Sigma$ by $\sigma$. Let the set of words with finite length be $\Sigma_*=\bigcup_{n=0}^{\infty}\left\{1,\dots,N\right\}^n$ with the convention that the only word of length $0$ is the empty word $\varnothing$. The set $\Sigma_n = \{ 1,\ldots,N \}^n$ is the collection of words of length $n$. Denote the length of $\bi \in \Sigma \cup \Sigma_*$ by $|\bi|$, and for finite or infinite words $\bi$ and $\bj$, let $\bi\wedge\bj$ be their common beginning. The concatenation of two words $\iii$ and $\jjj$ is denoted by $\iii\jjj$. We define the cylinder sets of $\Sigma$ in the usual way, that is, by setting
$$
  [\iii] = \{\jjj\in\Sigma : \iii \wedge \jjj = \iii \} = \{ \iii\jjj \in \Sigma : \jjj \in \Sigma \}
$$
for all $\iii \in \Sigma_*$. For a word $\bi=(i_1,\dots,i_n)$ with finite length let $f_{\bi}$ be the composition $f_{i_1}\circ\cdots\circ f_{i_n}$ and $A_{\bi}$ be the product $A_{i_1}\cdots A_{i_n}$. For $\bi\in\Sigma \cup \Sigma_*$ and $n < |\iii|$, let $\bi|_n$ be the first $n$ symbols of $\bi$. Let $\bi|_0=\varnothing$, $A_{\varnothing}$ be the identity matrix, and $f_{\varnothing}$ be the identity function.
Finally, we define the \textit{natural projection} $\pi=\pi_{\Av, \tv} \colon \Sigma \to E_{\Av,\tv}$ by setting
\begin{equation}\label{eq:natproj}
  \pi(\bi)=\sum_{k=1}^{\infty}A_{\bi|_{k-1}}v_{i_k}
\end{equation}
for all $\iii \in \Sigma$. Note that $E=\bigcup_{\bi\in \Sigma}\pi(\bi)$.

Denote by $\alpha_i(A)$ the $i$-th largest (counting with multiplicity) {\it singular value} of a matrix $A \in GL_d(\R)$, i.e.\ the positive square root of the $i$-th eigenvalue of $AA^T$, where $A^T$ is the transpose of $A$. We note that $\alpha_1(A)$ is the usual operator norm $\|A\|$ induced by the Euclidean norm on $\R^d$ and $\alpha_d(A)$ is the mininorm $\mathfrak{m}(A)=\|A^{-1}\|^{-1}$. We say that $A$ is \emph{contractive} if $\|A\|<1$. For a given tuple $\Av = (A_1,\ldots,A_N) \in GL_d(\R)^N$ we also set $\|\Av\| = \max_{i \in \{ 1,\ldots,N \}} \|A_i\|$ and $\mathfrak{m}(\Av) = \min_{i \in \{ 1,\ldots,N \}} \mathfrak{m}(A_i)$. Following Falconer \cite{Falconer88}, we define the {\it singular value function} $\phi^s$ of a matrix $A$ by setting
\begin{equation*}
  \phi^s(A) =
  \begin{cases}
    \alpha_1(A)\cdots\alpha_{\lfloor s\rfloor}(A)\alpha_{\lceil s\rceil}(A)^{s-\lfloor s\rfloor}, &\text{if } 0\leq s\leq d, \\
    |\det A|^{s/d}, &\text{if } s>d.
  \end{cases}
\end{equation*}
The singular value function satisfies
\begin{equation*}
  \fii^s(AB) \le \fii^s(A) \fii^s(B)
\end{equation*}
for all $A,B \in GL_d(\R)$. Moreover, if $(A_1,\ldots,A_N) \in GL_d(\R)^N$, then
\begin{equation} \label{eq:svf2}
  \fii^s(A_\iii)\mathfrak{m}(\Av)^{\delta |\iii|}
  \le \fii^{s+\delta}(A_\iii)
  \le \fii^s(A_\iii)\|\Av\|^{\delta |\iii|}
\end{equation}
for all $\iii \in \Sigma_*$ and $s,\delta \ge 0$.

For a tuple $\Av=(A_1,\dots, A_N)\in GL_d(\R)^N$ of contractive non-singular $d\times d$ matrices the function $P_{\Av}\colon [0,\infty)\to\R$ defined by
\begin{equation} \label{esap}
  P_{\Av}(s)=\lim_{n\rightarrow\infty}\tfrac{1}{n}\log\sum_{\iii \in \Sigma_n} \phi^s(A_{\iii})
\end{equation}
is called the \emph{singular value pressure}. It is well-defined, continuous, strictly decreasing on $[0,\infty)$, and convex between any two integers. Moreover, $P_{\Av}(0)=\log N$ and $\lim_{s\rightarrow\infty}P_{\Av}(s)=-\infty$. Let us denote by $\dimaff \Av$ the minimum of $d$ and the unique root of the singular value pressure function and call it the \emph{affinity dimension}.

If $\mu$ is a Radon measure on $\R^d$, then the \emph{upper} and \emph{lower local dimensions} of $\mu$ at $x$ are defined by
\begin{equation*}
  \udimloc(\mu,x) = \limsup_{r \downarrow 0} \frac{\log\mu(B(x,r))}{\log r} \quad \text{and} \quad \ldimloc(\mu,x) = \liminf_{r \downarrow 0} \frac{\log\mu(B(x,r))}{\log r},
\end{equation*}
respectively. The measure $\mu$ is \emph{exact-dimensional} if
\begin{equation*}
  \essinf_{x \sim \mu} \ldimloc(\mu,x) = \esssup_{x \sim \mu} \udimloc(\mu,x).
\end{equation*}
In this case, the common value is denoted by $\dim\mu$. The above quantities are naturally linked to set dimensions. For example, the \emph{lower Hausdorff dimension} of the measure $\mu$ is
\begin{equation*}
  \ldimh\mu = \essinf_{x \sim \mu} \ldimloc(\mu,x) = \inf\{ \dimh A : A \text{ is a Borel set with } \mu(A)>0 \}.
\end{equation*}
Here $\dimh A$ is the Hausdorff dimension of the set $A$. For more detailed information, the reader is referred to \cite{Falconer1997}.

Fix a probability vector $\pv=(p_1,\dots,p_N) \in (0,1)^N$ and denote the product $p_{i_1} \cdots p_{i_n}$ by $p_\iii$ for all finite words $\iii = (i_1,\ldots,i_n)$. Let $\nu_{\pv}$ be the corresponding \emph{Bernoulli measure} on $\Sigma$. It is uniquely defined by setting $\nu_{\pv}([\iii])=p_{\iii}$ for all $\iii \in \Sigma_*$. It is easy to see that $\nu_{\pv}$ is $\sigma$-invariant and ergodic. We say that $\nu$ on $\Sigma$ is a {\em step-$n$ Bernoulli measure} if it is a Bernoulli measure on $(\Sigma_n)^\N$ for some probability vector from $(0,1)^{N^n}$. Furthermore, we say that a measure $\nu$ on $\Sigma$ is {\em quasi-Bernoulli} if there is a constant $C \ge 1$ such that
\[
  C^{-1}\nu([\bi])\nu([\bj])\le \nu([\bi\bj])\le C\nu([\bi])\nu([\bj])
\]
for all $\bi, \bj\in\Sigma_*$. The \emph{entropy} of a $\sigma$-invariant measure $\nu$ on $\Sigma$ is
\begin{equation}\label{eq:entropy}
  h_{\nu}=-\lim_{n\to \infty}\tfrac 1n \sum_{\bi\in\Sigma_n}\nu([\bi])\log \nu([\bi]).
\end{equation}
Note that the entropy of a Bernoulli measure $\nu_{\pv}$ is given by $h_{\pv} = -\sum_{i=1}^Np_i\log p_i$.

If $\nu_\pv$ is a Bernoulli measure and $\Fb_{\Av,\tv}$ is an affine iterated function system, then the push-down measure $\mu_{\Av,\tv,\pv}=\pi_{\Av,\tv}\nu_{\pv}=\nu_{\pv}\circ(\pi_{\Av,\tv})^{-1}$ is called \textit{self-affine}. It is well known that the self-affine measure $\mu=\mu_{\Av,\tv,\pv}$ satisfies
\begin{equation*} 
  \mu = \sum_{i=1}^N p_i f_i\mu.
\end{equation*}
We say that $\Av\in GL_d(\R)^N$ satisfies the {\em totally dominated splitting condition} if there exist constants $C\ge 1$ and $0<\tau<1$ such that for every $i\in\{1,\dots,d-1\}$ either
\begin{equation}\label{eq:domsplit}
  \frac{\alpha_{i+1}(A_{\iii})}{\alpha_{i}(A_{\iii})}\le C\tau^{|\iii|}
\end{equation}
for every $\iii\in\Sigma_*$ or
$$
  \frac{\alpha_{i+1}(A_\bi)}{\alpha_i(A_\bi)}>C^{-1}
$$
for every $\iii\in\Sigma_*$. By Bochi and Gourmelon \cite[Theorem~B]{BochiGourmelon}, the set
\begin{equation} \label{eq:d}
  \mathcal{D} = \{\Av\in GL_d(\R)^N : \text{ \eqref{eq:domsplit} holds for every }i\in\{1,\dots,d\}\}.
\end{equation}
is an open subset of $GL_d(\R)^N$.

If $\nu$ is an ergodic probability measure on $\Sigma$, then, by Oseledets' theorem, there exist constants $0<\chi_{1}(\Av, \nu)\leq\cdots\leq\chi_{d}(\Av, \nu)<\infty$ such that
\begin{equation} \label{eq:lyapexp}
  \chi_i(\Av, \nu) = -\lim_{n\to\infty}\tfrac{1}{n}\log\alpha_i(A_{i_1}\cdots A_{i_n})
\end{equation}
for $\nu$-almost every $\bi\in \Sigma$. The numbers $\chi_{i}(\Av, \nu)$ are called the \emph{Lyapunov exponents} of $\Av$ with respect to $\nu$. The Lyapunov exponents of
$\Av$ with respect to a Bernoulli measure $\nu_{\pv}$ are denoted by $\chi_{i}(\Av, \pv)$. Furthermore, let us define the \emph{Lyapunov dimension} of $\nu$ by
$$
 \diml\nu = \min_{k\in\{0, \dots, d\}}\biggl\{k + \frac{h_{\nu}-\sum_{j=1}^{k} \chi_j(\Av,\nu)}{\chi_{k+1}(\Av,\nu)},d\biggr\}.
$$
The Lyapunov dimension of the projected measure $\pi \nu$ on $E_{\Av,\tv}$ is defined by setting $\diml\pi\nu=\diml\nu$. K\"aenm\"aki \cite[Theorems 2.6 and 4.1]{Kaenmaki04}
proved the existence of ergodic equilibrium states. If $s = \dimaff\Av \le d$, then an ergodic $s$-equilibrium state $\mu$ of $\Av$ on $\Sigma$ is defined by the equality
\begin{equation} \label{eq:eq-state}
  h_\mu = -\lim\tfrac{1}{n}\sum_{\iii\in\Sigma_n}\mu([\iii])\log\fii^s(A_\iii) = \sum_{j=1}^{\lfloor s \rfloor} \chi_j(\Av,\mu) + (s-\lfloor s \rfloor)\chi_{\lceil s \rceil}(\Av,\mu).
\end{equation}
Here the second equality follows from Kingman's ergodic theorem.
It is easy to see that such an $s$-equilibrium state is a measure of maximal Lyapunov dimension.

B\'ar\'any and K\"aenm\"aki recently proved in \cite[Theorem~2.3]{BaranyKaenmaki} that the self-affine measure $\mu_{\Av,\tv,\pv}$ is exact-dimensional regardless of the choices of $\Av$, $\tv$, and $\pv$ provided that all the corresponding Lyapunov exponents are distinct. Furthermore, they showed that if $\Av$ satisfies the totally dominated splitting condition, then the image of any quasi-Bernoulli measure under $\pi_{\Av,\tv}$ is exact-dimensional regardless of the choice of $\tv$; see \cite[Theorem~2.6]{BaranyKaenmaki}.

Let us next state the main results of the article. Let $\mathcal L^d$ be the $d$-dimensional Lebesgue measure and $\udimm$ be the upper Minkowski dimension. We define $\|\tv\|=\max_{i\in\{1, \dots, N\}}|v_i|$ and recall that $\|\Av\|=\max_{i\in\{1, \dots, N\}}\|A_i\|$.

\begin{maintheorem}\label{thm:mainplanar}
  Suppose that $\tv=(v_1,\dots, v_N)\in(\R^2)^N$ is such that $v_i\neq v_j$ for $i\neq j$ and
  \begin{equation}\label{eq:matricesplanar}
    \mathcal{A}_{\tv}=\biggl\{\Av\in GL_2(\R)^N : 0<\max_{i\neq j}\frac{\|A_i\|+\|A_j\|}{|v_i-v_j|}\cdot\frac{\|\tv\|}{1-\|\Av\|}<\frac{\sqrt{2}}{2}\biggr\}.
  \end{equation}
  Then
  $$
    \dimh E_{\Av,\tv}=\udimm E_{\Av,\tv}=\dimaff\Av
  $$
  for $\mathcal{L}^{4N}$-almost all $\Av\in\mathcal{A}_{\tv}$. Moreover, for every probability vector $\pv \in (0,1)^N$ the corresponding self-affine measure $\mu=\mu_{\Av,\tv,\pv}$ satisfies
  $$
    \dim\mu=\diml\mu=\min\left\{\frac{h_{\pv}}{\chi_1(\Av,\pv)},1+\frac{h_{\pv}-\chi_1(\Av,\pv)}{\chi_2(\Av,\pv)}\right\}
  $$
  for $\mathcal{L}^{4N}$-almost all $\Av\in\mathcal{A}_{\tv}$.
\end{maintheorem}

Note that matrix tuples in $\mathcal{A}_{\tv}$ are contractive. In higher dimensions, our result is the following.

\begin{maintheorem}\label{thm:main}
  Suppose that $d\in \N$ is such that $d\ge 3$, $\tv=(v_1,\dots, v_N)\in(\R^d)^N$ is such that $v_i\neq v_j$ for $i\neq j$, and
  \begin{equation}\label{eq:matrices}
    \mathcal{A}_{\tv}'=\biggl\{\Av\in GL_d(\R)^N : 0<\max_{i\neq j}\frac{\|A_i\|+\|A_j\|}{|v_i-v_j|}\cdot\frac{\|\tv\|}{1-\|\Av\|}<\frac{2}{\sqrt{3}}-1\biggr\}.
  \end{equation}
  Then for every probability vector $\pv \in (0,1)^N$ the corresponding self-affine measure $\mu=\mu_{\Av,\tv,\pv}$ satisfies
  $$
    \dim\mu=\diml\mu=\min_{k\in\{0, \dots, d\}}\biggl\{k + \frac{h_{\pv}-\sum_{j=1}^{k} \chi_j(\Av,\pv)}{\chi_{k+1}(\Av,\pv)},d\biggr\}
  $$
  for $\mathcal{L}^{d^2N}$-almost all $\Av\in\mathcal{A}_{\tv}'$. Moreover,
  $$
    \dimh\mu_{\Av}=\dimh E_{\Av,\tv}=\udimm E_{\Av,\tv}=\dimaff\Av
  $$
  for $\mathcal{L}^{d^2N}$-almost all $\Av\in\mathcal{A}_{\tv}'\cap\mathcal{D}$, where $\mathcal{D}$ is as in \eqref{eq:d} and $\mu_{\Av}$ is an ergodic $s$-equilibrium state of $\Av$ for $s = \dimaff\Av$.
\end{maintheorem}

Let us show that the equilibrium states in Theorem \ref{thm:main} are quasi-Bernoulli.

\begin{lemma}\label{lem:quasi}
  For $\mathcal{L}^{d^2N}$-almost every $\Av\in\mathcal A'_{\tv}\cap\mathcal{D}$ the unique $s$-equilibrium state of $\Av$ is quasi-Bernoulli for $s=\dimaff (\Av)$.
\end{lemma}

\begin{proof}
  By \cite[Propositions 3.4 and 3.6]{KaenmakiBing}, for $\mathcal{L}^{d^2N}$-almost every $\Av\in GL_d(\R)^N$, the $s$-equilibrium state of $\Av$ for $s=\dimaff\Av$ is unique and satisfies the following Gibbs property: there exists a constant $C \ge 1$ such that
  $$
    C^{-1}\phi^s(A_{\ii})\leq\mu([\ii])\leq C\phi^s(A_{\ii})
  $$
  for all $\ii\in\Sigma_*$.
  By \cite[Theorem~B]{BochiGourmelon} and \cite[Lemma~2.1]{FengShmerkin14}, for each $\Av\in\mathcal{D}$, there exists a constant $C'>0$ such that
  $$
    C'\phi^s(A_{\ii})\phi^s(A_{\jj})\leq\phi^s(A_{\ii\jj})
  $$
  for all $\ii,\jj\in\Sigma_*$. The statement of the lemma follows.
\end{proof}

\begin{remark} \label{rem:diff-proofs}
  We emphasize that the methods used to prove Theorems \ref{thm:mainplanar} and \ref{thm:main} are significantly different. At first, the higher dimensional exact
  dimensionality result for quasi-Bernoulli measures (the Ledrappier-Young formula to be more precise) of B\'ar\'any and K\"aenm\"aki \cite{BaranyKaenmaki} requires the totally dominated splitting condition, which is why we restrict our matrix tuples to the set $\mathcal{D}$. Very recently, Feng [personal communication] has informed the authors that the Ledrappier-Young formula holds also without totally dominated splitting. By relying on this, one could improve Theorem \ref{thm:main} by replacing $\mathcal{A}_{\tv}' \cap \mathcal{D}$ by $\mathcal{A}_{\tv}'$.
  Secondly, the transversality argument used in the higher dimensional case is different from the two-dimensional case. Curiously, the higher
  dimensional transversality argument requires the dimension to be at least three, so it cannot be used in the two-dimensional case. This difference is also the reason
  why we use different upper bounds in the definitions of $\mathcal{A}_{\tv}$ and $\mathcal{A}_{\tv}'$. Currently we do not know if the upper bound $2/\sqrt{3}-1$ used
  in \eqref{eq:matrices} can be replaced by the upper bound $\sqrt{2}/2$ used in \eqref{eq:matricesplanar}. The sharpness of the methods used in our proofs is discussed in
  Remark \ref{ex:nonsharpness}.
\end{remark}

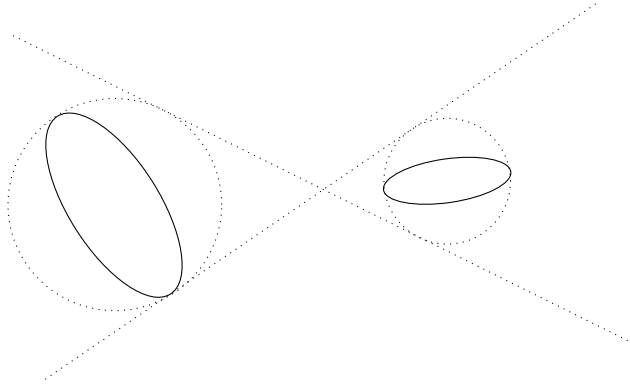
\begin{figure}[t]
  \begin{tikzpicture}[scale=1.8]
    \draw[dotted] (2.4600934554196052,4.059391729465462)-- (6.954596872214569,1.8121400210679846);
    \draw[dotted] (2.6950056607262165,1.5286867989015813)-- (6.729871701374821,4.301403451908268);
    \draw[dotted] (3.20341517435108,2.8147600140453206) circle (0.781178222586814cm);
    \draw[dotted] (5.632925394859022,2.987228510882403) circle (0.46280121042343775cm);
    \draw [rotate around={8.199457513440423:(5.6324828428403855,2.9902998151234286)}] (5.6324828428403855,2.9902998151234286) ellipse (0.47005038392683764cm and 0.16116465859330437cm);
    \draw [rotate around={-57.828635740821696:(3.1969078950561722,2.8106667021345255)}] (3.1969078950561722,2.8106667021345255) ellipse (0.7741530957060199cm and 0.33042905168588915cm);
  \end{tikzpicture}
  \caption{The condition \eqref{eq:matricesplanar} requires that the images of the unit ball under $f_i$ are separated away from each other as in the picture: if the
  ellipses are rotated around their centers, they still stay inside a cone having an angle at most $\pi/2$.}
  \label{fig:illustration}
\end{figure}

\begin{remark} \label{rem:remark2.2}
  Many of the recent works on dimensions of self-affine sets and measures (see e.g.\ \cite{BaranyKaenmaki, barany2015dimension, MorrisShmerkin2016, Rapaport2015}) rely on properties of the
  Furstenberg measure (for definitions, see \S \ref{sec:transversality})  and on the exceptional sets of the dimension of orthogonal projections. We remark that the result of B\'ar\'any and Rams \cite{barany2015dimension} is the first result in the direction of almost every matrices. However, Theorem \ref{thm:mainplanar} covers situations that cannot be addressed by using this
  approach. Even though the condition \eqref{eq:matricesplanar} is rather restrictive (for example, we will see in Lemma \ref{lem:ssc-d} that it implies that the images
  $f_i(E)$ are disjoint, i.e.\ $\Fb_{\Av, \tv}$ satisfies the {\em strong separation condition}) Theorem \ref{thm:mainplanar} introduces a checkable condition for an affine
  iterated function system to satisfy the desired dimension result. This is in contrast to, for example, \cite[Corollaries 2.7 and 2.8]{BaranyKaenmaki} where the claim for
  the self-affine measure $\mu$ in Theorem \ref{thm:mainplanar} holds provided that the strong separation condition holds and the dimension of $\mu$ does not drop when
  projected to orthogonal complements of Furstenberg typical lines. The condition \eqref{eq:matricesplanar} can be illustrated via Lemma \ref{lem:equivalence} as in
  Figure \ref{fig:illustration}.

  We will next exhibit an open set of matrices that satisfy the assumptions of Theorem \ref{thm:mainplanar}, but do not satisfy the projection condition of
  \cite{BaranyKaenmaki}. Consider an affine iterated function system consisting of three mappings $x \mapsto A_ix+v_i$. Let the translation vectors $v_i$ be
  equidistributed on the unit circle, that is, $|v_i|=1$ and $|v_i-v_j| = \sqrt{3}$ for all $i \ne j$. It is easy to see that if $\|A_i\|<\sqrt{6}/(4+\sqrt{6}) =: \eta$ for all $i$, then $\Av \in \mathcal{A}_{\tv}$. Since $\eta > 1/3$ we find an open set of three matrices such that all elements are strictly positive, the Furstenberg measure is supported on a Cantor set having dimension less than $1/2$, and the affinity dimension is between $1$ and $3/2$. Recalling \cite[Corollary 2.9]{BaranyKaenmaki}, we see that this case does not satisfy the projection condition. Also, this example is outside of the scope of Morris and Shmerkin \cite[Theorems~1.2 and 1.3]{MorrisShmerkin2016}, since the affinity dimension is less than $3/2$ and, as the affinity dimension is strictly larger than the dimension of the Furstenberg measure, the bunching condition does not hold either. 
\end{remark}

In the next lemma, we note that the strong separation condition follows from the conditions \eqref{eq:matricesplanar} and \eqref{eq:matrices}. Denote by $O(d)$ the orthogonal
group of matrices $G\in GL_d(\R)$ with $G^TG=I$. For $G\in O(d)$, we denote the vector $(Gv_1, \dots, Gv_d)$ by $G(\tv)$.

\begin{lemma}\label{lem:ssc-d}
  Let $\tv = (v_1,\dots, v_N) \in (\R^d)^N$ be such that $v_i\neq v_j$ for $i\neq j$ and $\Av = (A_1,\dots,A_N) \in \mathcal{A}_{\tv} \cup \mathcal A_\tv'$. The affine iterated function system $\Fb_{\Av,G(\tv)}$ satisfies the strong separation condition for all $G\in O(d)$.
\end{lemma}

\begin{proof}
  By definition, it is easy to see that
  $$
    |\pi_{\Av, G(\tv)}(\ii)| \leq \frac{\|\tv\|}{1-\|\Av\|}
  $$
  for every $\ii\in\Sigma$ and $G\in O(d)$. If $\bi,\bj\in\Sigma$ are such that $\bi|_1\neq \bj|_1$, then, by either \eqref{eq:matricesplanar} or \eqref{eq:matrices},
  \begin{align*}
    |\pi_{\Av, G(\tv)}(\bi)-\pi_{\Av, G(\tv)}(\bj)|&\geq|Gv_{i_1}-Gv_{j_1}|-|A_{i_1}\pi_{\Av, G(\tv)}(\sigma(\bi))-A_{j_1}\pi_{\Av, G(\tv)}(\sigma(\bj))|\\
    &\ge |v_{i_1}-v_{j_1}|-(\|A_{i_1}\|+\|A_{j_1}\|)\frac{\|\tv\|}{1-\|\Av\|}>0.
  \end{align*}
  This shows that the strong separation condition holds in both cases, proving the claim.
\end{proof}

Very recently, after this article was finished, B\'ar\'any, Hochman, and Rapaport \cite{baranyhochmanrapaport} showed that the dimension of planar self-affine measures is equal to the Lyapunov dimension if the strong open set condition holds and the matrix tuple is strongly irreducible. 
Furthermore, we point out that, under the assumption that $(v_1,\ldots,v_N)$ is linearly independent (which, in particular, forces $N\leq d$), one can prove the result of Theorem~\ref{thm:main} for every ergodic measure.

\begin{prop}\label{prop:report}
  Suppose that $\tv=(v_1,\ldots,v_N)\in(\R^d)^N$ is linearly independent and
  \begin{equation*}
    \mathcal{A}_t = \{(A_1,\ldots,A_N)\in GL_d(\R)^N:\|A_i\|<t\text{ for all }i \in \{1,\ldots,N\}\}
  \end{equation*}
  for all $t > 0$. Then every ergodic probability measure $\nu$ on $\Sigma$ satisfies
  \begin{equation*}
    \dim\pi_{\Av,\tv}\nu=\diml\nu=\min_{k\in\{0, \dots, d\}}\biggl\{k + \frac{h_{\pv}-\sum_{j=1}^{k} \chi_j(\Av,\nu)}{\chi_{k+1}(\Av,\nu)},d\biggr\}
  \end{equation*}
  for $\mathcal{L}^{d^2N}$-almost all $\Av \in \mathcal{A}_{1/2}$. Moreover,
  \begin{equation*}
    \dimh E_{\Av,\tv}=\udimm E_{\Av,\tv}=\dimaff\Av
  \end{equation*}
  for $\mathcal{L}^{d^2N}$-almost all $\Av \in \mathcal{A}_{1/2}$.
\end{prop}

\begin{proof}
  By Jordan, Pollicott, and Simon \cite[Theorem~1.9]{JordanPollicottSimon07}, for every $\Av\in\mathcal{A}_{1/2}$ there exists $X_\Av\subset(\R^d)^N$ such that
  $$
    \mathcal{L}^{dN}((\R^d)^N\setminus X_{\Av})=0,
  $$
  and for every $\tv\in X_\Av$, $\dim\pi_{\Av,\tv}\nu=\diml\nu$. Thus, by applying Fubini's Theorem on the space $(\R^d)^N\times\mathcal{A}_{1/2}$ with the measure $\mathcal{L}^{dN}\times\mathcal{L}^{d^2N}$, there exists $X\subset(\R^d)^N$ such that $\mathcal{L}^{dN}((\R^d)^N\setminus X)=0,$ and for every $\tv\in X$ there exists $Y_\tv\subset\mathcal{A}_{1/2}$ such that $\mathcal{L}^{d^2N}(\mathcal{A}_{1/2}\setminus Y_\tv)=0$, and for every $\Av\in Y_\tv$, $\dim\pi_{\Av,\tv}\nu=\diml\nu$.

  Let $\tv=(v_1,\ldots,v_N)$ be linearly independent. Thus, for any linearly independent $\tv'\in(\R^d)^N$, there exists a unique $B\in GL_d(\R)$ such that $B(\tv)=(Bv_1,\ldots,Bv_N)=\tv'$. Moreover,
  $$
    \dim\pi_{B^{-1}(\Av),\tv}\nu=\dim\pi_{\Av,B(\tv)}\nu,
  $$
  for all $\Av\in GL_d(\R^d)^N$, where $B^{-1}(\Av)=(B^{-1}A_1B,\ldots,B^{-1}A_NB)$. To verify this claim, consult Lemma~\ref{lem:equivalence}. It is easy to see that $B^{-1}\colon GL_d(\R^d)^N \to GL_d(\R^d)^N$ is a bi-Lipschitz mapping and that
  $$
    \mathcal{A}_{\alpha_d(B)t/\alpha_1(B)}\subset B^{-1}(\mathcal{A}_t)\subset \mathcal{A}_{\alpha_1(B)t/\alpha_d(B)}.
  $$
  Hence, for every $\tv'\in X$ there exists $B\in GL_d(\R^d)$ such that $B(\tv)=\tv'$ and thus, for every $\Av\in Y_{B(\tv)}$ it holds that
  $$
    \dim\pi_{B^{-1}(\Av),\tv}\nu=\diml\nu.
  $$
  In particular, $\dim\pi_{\Av,\tv}\nu=\diml\nu$ for $\mathcal{L}^{d^2N}$-almost all $\Av\in\mathcal{A}_{\alpha_d(B)/(2\alpha_1(B))}$. Since $X$ has full measure, and therefore is dense, there exists a sequence of such $B$'s converging to the identity, which implies the first half of the assertion. The second half follows by a similar argument.
\end{proof}

\section{Modified self-affine transversality}\label{sec:transversality}

This section is devoted to proving that for measures, which satisfy the Ledrappier-Young formula, the dimension of the measure is equal to the Lyapunov dimension for almost every matrix tuple whenever the modified self-affine transversality condition, defined below, holds.

Let us denote the Grassmannian of $k$-planes in $\R^d$ by $G(d,k)$.  If $V \in G(d,k)$, then $V^\bot \in G(d,d-k)$ is the subspace orthogonal to $V$.  For $A\in GL_d(\R)$ let $\|A|V\|$ be the operator norm of $A$ restricted to $V$ defined by $\|A|V\|=\sup_{v\in V}|Av|/|v|$ and $\mathfrak{m}(A|V)$ the mininorm of $A$ restricted to $V$ defined by $\mathfrak{m}(A|V)=\inf_{v\in V}|Av|/|v|$.

Let $\nu$ be an ergodic $\sigma$-invariant measure on $\Sigma$ and let $\Av=(A_1,\dots,A_N) \in GL_d(\R)^N$. We say that $\mu_F^{d,k}$ is the \emph{$(d,k)$-Furstenberg-measure} with respect to $\nu$ and $\Av$ if $\mu_F^{d,k} \times \nu$ is an ergodic $T$-invariant measure, where
$$
  T\colon G(d,k)\times\Sigma \to G(d,k)\times\Sigma, \quad (V,\ii) \mapsto (A_{i_0}^{-1}V,\sigma(\ii)),
$$
and furthermore,
\begin{equation}\label{eq:Furstlimit}
  \lim_{n\to\infty}\tfrac{1}{n}\log\mathfrak{m}(A_{i_n}^{-1}\cdots A_{i_0}^{-1}|V)=\chi_{d-k+1}(\Av,\nu)
\end{equation}
for $\mu_F^{d,k}\times\nu$-almost all $(V,\ii) \in G(d,k)\times\Sigma$.

\begin{lemma}\label{lem:projection_lyapunov}
  Let $\nu$ be an ergodic measure on $\Sigma$ and let $\Av = (A_1,\ldots,A_N) \in GL_d(\R)^N$ be a tuple of contractive matrices. Assume that the $(d,d-k)$-Furstenberg measure $\mu_F^{d,d-k}$ with respect to $\nu$ and $\mathbf{A}$ exists for some $k\in\{1,\dots,d-1\}$. Then
  \begin{equation*}
    -\lim_{n\to\infty}\tfrac{1}{n} \log \fii^s(\proj_{V^\perp}A_{\iii|_n}) = \sum_{j=1}^{\lfloor s \rfloor} \chi_j(\Av,\nu) + (s-\lfloor s \rfloor)\chi_{\lceil s \rceil}(\Av,\nu)
  \end{equation*}
  for all $0\le s\leq k$ and for $\mu_F^{d-k} \times \nu$-almost all $(V,\iii) \in G(d,d-k) \times \Sigma$.
\end{lemma}

\begin{proof}
  Notice that if $A \in GL_d(\R)$ and $\ell<s\le \ell+1$, then $\phi^s(A)=\|A^{\wedge \ell}\|^{\ell+1-s}\|A^{\wedge(\ell+1)}\|^{s-\ell}$ and $\phi^s(A)=\phi^s(A^T)$; for example, see \cite[\S 3.4]{KaenmakiMorris2016}. Thus, it is enough to show that
  \begin{equation*} 
    -\lim_{n\to\infty}\tfrac{1}{n}\log\| ((A_{\iii|_n})^T\proj_{V^\perp})^{\wedge\ell} \|=\sum_{j=1}^{\ell} \chi_j(\Av,\nu)
  \end{equation*}
  and the corresponding limit for $\ell+1$ hold
  for $\mu_F^{d-k} \times \nu$-almost all $(V,\iii) \in G(d,d-k) \times \Sigma$. It is easy to see that $(\proj_{V^{\perp}})^{\wedge\ell}=\proj_{\wedge^{\ell}V^{\perp}}$ and $\| ((A_{\iii|_n})^T)^{\wedge \ell}\proj_{\wedge^{\ell}V^{\perp}} \|=\|((A_{\iii|_n})^T)^{\wedge\ell}|\wedge^{\ell}V^{\perp}\|$. Since $k \ge \ell+1$, it suffices to show that
  \begin{equation}\label{eq:exterlim}
    -\lim_{n\to\infty}\tfrac{1}{n}\log\| ((A_{\iii|_n})^T)^{\wedge k}|\wedge^{k}V^{\perp}\|=\sum_{j=1}^{k} \chi_j(\Av,\nu)
  \end{equation}
  for $\mu_F^{d-k} \times \nu$-almost all $(V,\iii) \in G(d,d-k) \times \Sigma$. By the Oseledets' decomposition, this implies that the corresponding limit holds for all values in $\{1,\ldots,k\}$.

  We denote the Hodge star operator between $\wedge^k\R^d$ and $\wedge^{d-k}\R^d$ by $*$. Let $\{e_1,\ldots,e_d\}$ be the standard orthonormal basis of $\R^d$. The operator $*$ is the bijective linear map satisfying
  \begin{equation*}
    *(e_{i_1} \wedge \cdots \wedge e_{i_k}) = \sgn(i_1,\ldots,i_d) e_{i_{k+1}} \wedge \cdots \wedge e_{i_d}
  \end{equation*}
  for all $1 \le i_1 < \cdots < i_k \le d$, where $1 \le i_{k+1} < \cdots < i_d \le d$ are such that $\{ i_{k+1},\ldots,i_d \} = \{1,\ldots,d\} \setminus \{i_1,\ldots,i_k\}$, and $\sgn(i_1,\ldots,i_d) = 1$ if $(i_1,\ldots,i_d)$ is an even permutation of $\{1,\ldots,d\}$ and $\sgn(i_1,\ldots,i_d) = -1$ otherwise. Recall that the inner product on $\wedge^k\R^d$ is defined by setting $\la v,w \ra = *(v \wedge *w)$ for all $v,w \in \wedge^k\R^d$. It is straightforward to see that $*(*v) = (-1)^{k(d-k)}v$ and $v \wedge *w = (-1)^{k(d-k)} *\!v \wedge w$ for all $v,w \in \wedge^k\R^d$ and hence, $\|v\| = \|*\!v\|$ for all $v \in \wedge^k\R^d$. For a more detailed treatment, the reader is referred e.g.\ to \cite[\S 3.2]{KaenmakiMorris2016}.


  Observe first that if $A \in GL_d(\R)$ and the vectors $v,w\in\R^d$ are perpendicular, then also the vectors $A^Tv,A^{-1}w \in \R^d$ are perpendicular. Let $v_1,\dots,v_{d-k}$ be an orthonormal basis of $V$ and $v_{d-k+1},\dots,v_d$ be an orthonormal basis of $V^{\perp}$. Moreover, let $g_1,\dots,g_{d-k}$ and $g_1',\dots,g_k'$ be the vectors obtained from the Gram-Schmidt orthogonalization of $(A_{\ii|_n})^{-1}v_1,\dots, (A_{\ii|_n})^{-1}v_{d-k}$ and $(A_{\ii|_n})^Tv_{d-k+1},\dots,(A_{\ii|_n})^Tv_d$, respectively. This means that
  \begin{align*}
    g_1 &= (A_{\ii|_n})^{-1}v_1, \\
    g_2 &= (A_{\ii|_n})^{-1}v_1-c_{1,1} g_1, \\
    &\;\;\vdots \\
    g_{d-k} &= (A_{\ii|_n})^{-1}v_{d-k}-c_{d-k,1}g_1-\cdots-c_{d-k,d-k-1}g_{d-k-1}
  \end{align*}
  and
  \begin{align*}
    g_1' &= (A_{\ii|_n})^Tv_{d-k+1}, \\
    g_2' &= (A_{\ii|_n})^Tv_{d-k+2}-c_{1,1}'g_1', \\
    &\;\;\vdots \\
    g_k' &= (A_{\ii|_n})^Tv_{d}-c_{k,1}'g_1'-\cdots-c_{k,k-1}g_{k-1}'
  \end{align*}
  with appropriate choices of the constants $c_{i,j}$ and $c_{i,j}'$. Hence,
  $$
    (A_{\ii|_n})^{-1}v_1\wedge\cdots\wedge (A_{\ii|_n})^{-1}v_{d-k}=g_1\wedge\cdots\wedge g_{d-k}
  $$
  and
  $$
    (A_{\ii|_n})^Tv_{d-k+1}\wedge\cdots \wedge (A_{\ii|_n})^Tv_d=g_1'\wedge\cdots\wedge g_k',
  $$
  and $\{g_1,\dots,g_{d-k},g_1',\dots,g_k'\}$ is an orthogonal basis of $\R^d$.
  Therefore,
  \begin{align*}
    0 \neq (g_1\wedge\cdots\wedge g_{d-k}) \wedge (g_1'\wedge\cdots\wedge g_k') = \inner{g_1\wedge\cdots\wedge g_{d-k}}{*(g_1'\wedge\cdots\wedge g_k')}e_1\wedge\cdots\wedge e_d.
  \end{align*}
  Since any $g_i'$ is perpendicular to any $g_j$, we must have
  $$
    *(g_1'\wedge\cdots\wedge g_k') = C g_1\wedge\cdots\wedge g_{d-k}
  $$
  for some constant $C>0$. Since
  \begin{align*}
    \det(A_{\ii|_n})^T&v_1\wedge\cdots\wedge v_d
    = ((A_{\ii|_n})^Tv_{1}\wedge\cdots \wedge (A_{\ii|_n})^Tv_{d-k}) \wedge ((A_{\ii|_n})^Tv_{d-k+1}\wedge\cdots \wedge (A_{\ii|_n})^Tv_d) \\
    &= \inner{(A_{\ii|_n})^Tv_{1}\wedge\cdots \wedge (A_{\ii|_n})^Tv_{d-k}}{*((A_{\ii|_n})^Tv_{d-k+1}\wedge\cdots \wedge (A_{\ii|_n})^Tv_d)}e_1\wedge\cdots\wedge e_d \\
    &= C\inner{(A_{\ii|_n})^Tv_{1}\wedge\cdots \wedge (A_{\ii|_n})^Tv_{d-k}}{(A_{\ii|_n})^{-1}v_1\wedge\cdots\wedge (A_{\ii|_n})^{-1}v_{d-k}}e_1\wedge\cdots\wedge e_d \\
    &= C\inner{v_1\wedge\cdots\wedge v_{d-k}}{v_1\wedge\cdots\wedge v_{d-k}}e_1\wedge\cdots\wedge e_d \\
    &= Ce_1\wedge\cdots\wedge e_d
  \end{align*}
  we have $|C|=|\det(A_{\ii|_n})^T|$. Therefore, by \eqref{eq:Furstlimit} and the Oseledets' decomposition,
  \begin{align*}
    \sum_{j=k+1}^{d} \chi_j(\Av,\nu)
    &= \lim_{n\to\infty}\tfrac{1}{n}\log\| ((A_{\iii|_n})^{-1})^{\wedge (d-k)}|\wedge^{d-k}V \| \\
    &= \lim_{n\to\infty}\tfrac{1}{n}\log\| (A_{\iii|_n})^{-1}v_1\wedge\cdots\wedge(A_{\iii|_n})^{-1}v_{d-k} \|\\
    &= \lim_{n\to\infty}\tfrac{1}{n}\log\frac{\|*\!((A_{\ii|_n})^Tv_{d-k+1}\wedge\cdots \wedge (A_{\ii|_n})^Tv_d)\|}{\det(A_{\ii|_n})^T}\\
    &= \lim_{n\to\infty}\tfrac{1}{n}\log\|(A_{\ii|_n})^Tv_{d-k+1}\wedge\cdots \wedge (A_{\ii|_n})^Tv_d\| + \sum_{j=1}^{d} \chi_j(\Av,\nu)\\
    &= \lim_{n\to\infty}\tfrac{1}{n}\log\|((A_{\ii|_n})^T)^{\wedge k}|\wedge^kV\|+\sum_{j=1}^{d} \chi_j(\Av,\nu)
  \end{align*}
  for $\mu_F^{d-k} \times \nu$-almost every $(V,\iii) \in G(d,d-k) \times \Sigma$. This completes the proof of the lemma.
\end{proof}


Let $\Av\in GL_d(\R)^N$ be a tuple of contractive matrices and $\tv = (v_1,\ldots,v_N) \in (\R^d)^N$. Let $\mathcal U$ be a parameter space equipped with a measure $m$ such that each $u\in \mathcal U$ is a mapping $u\colon \R^d\to \R^d$. We will use this parametrised family of mappings to modify the self-affine iterated function system $\Fb_{\Av, \tv}$ by replacing it with $\Fb_{\Av, u(\tv)}$, where $u(\tv)=(u(v_1), \dots, u(v_N))$. We say that the pair $(\mathcal U, m)$ satisfies the \emph{modified self-affine transversality condition} for $\Av$, if there exists a constant $C > 0$ such that for every proper subspace $V$ of $\R^d$ and $t>0$ it holds that
\begin{equation*} 
  m(\{ u \in \mathcal{U} : |\proj_V\pi_{\Av, u(\tv)}(\iii)-\proj_V\pi_{\Av, u(\tv)}(\jjj)| < t \}) \le C\prod_{i=1}^{\dim V} \min\biggl\{ 1,\frac{t}{\alpha_i(\proj_VA_{\iii\wedge\jjj})} \biggr\}
\end{equation*}
for all $\iii,\jjj \in \Sigma$ with $\iii \ne \jjj$.
We note that if $V=\R^d$, then this condition is the self-affine transversality condition of Jordan, Pollicott and Simon \cite{JordanPollicottSimon07}.

\begin{lemma}\label{lem:transint}
  Let $\tv \in (\R^d)^N$ and $\Av = (A_1,\ldots,A_N) \in GL_d(\R)^N$ be a tuple of contractive matrices. If $(\mathcal U, m)$ satisfies the modified transversality condition for $\Av$, then there exists a constant $C>0$ such that
  \begin{align*}
    \int \frac{\mathrm{d}m(u)}{|\proj_{V}(\pi_{\Av, u(\tv)}(\iii))-\proj_{V}(\pi_{\Av, u(\tv)}(\jjj))|^s} \le \frac{C}{\fii^s(\proj_{V}A_{\iii\wedge\jjj})}
  \end{align*}
  for every proper subspace $V$ of $\R^d$ and for all $0\le s<\dim V$.
\end{lemma}

\begin{proof}
  The proof is a slight modification of the proof of \cite[Lemma~4.5]{JordanPollicottSimon07} and hence omitted.
\end{proof}

We will use the above lemma in the proof of the following proposition which is a key observation related to the modified self-affine transversality condition.

\begin{proposition} \label{thm:projection}
  Let $\tv \in (\R^d)^N$, $\Av = (A_1,\ldots,A_N) \in GL_d(\R)^N$ be a tuple of contractive matrices, and $\nu$ be an ergodic measure on $\Sigma$. Assume that the $(d,d-k)$-Furstenberg measure $\mu_F^{d,d-k}$ with respect to $\nu$ and $\Av$ exists for all $k\in\{1,\dots,d-1\}$. If $(\mathcal U, m)$ satisfies the modified self-affine transversality condition for $\Av$, then
  \begin{equation*}
    \ldimh \proj_{V^\perp}\pi_{\Av, u(\tv)}\nu = \min\{ k,\diml\nu \}
  \end{equation*}
  for $\mu_F^{d,d-k}$-almost all $V \in G(d,d-k)$, for all $k \in \{ 1,\ldots,d-1 \}$, and for $m$-almost all $u \in \mathcal{U}$.
\end{proposition}

\begin{proof}
  To simplify notation, we denote $\chi_j(\Av,\nu)$ by $\chi_j$ and $\pi_{\Av, u(\tv)}$ by $\pi_u$. Let $s < \min\{ \diml\nu,k \}$. By standard methods,
  it suffices to show that for $\nu$-almost every $\iii \in \Sigma$ it holds that
  \begin{equation*}
    \int |\proj_{V^{\perp}}\pi_u(\iii)-\proj_{V^{\perp}}\pi_u(\jjj)|^{-s} \dd\nu(\jjj)  < \infty
  \end{equation*}
  for $\mu^{d,d-k}_F$-almost all $V \in G(d,d-k)$ and $m$-almost all $u \in \mathcal{U}$. For this to hold, it is enough to prove that for every $\eps>0$ small enough there exist sets $E_1\subset \Sigma$ with $\nu(E_1)>1-\eps$ and $E_2\subset G(d,d-k)\times\Sigma$ with $\mu^{d-k}_F\times\nu(E_2)>1-\varepsilon$ such that
  \begin{equation*}
    \iiint \frac{\mathrm{d}\nu|_{E_1}(\iii)\,\mathrm{d}(\mu_F^{d,d-k}\times\nu)|_{E_2}(V,\jjj)\,\mathrm{d}m(u)}{|\proj_{V^{\perp}}\pi_u(\iii)-\proj_{V^{\perp}}\pi_u(\jjj)|^s} < \infty.
  \end{equation*}
  Fix $0<\eps<\chi_{\lceil s \rceil}(\diml\nu-s)/2$. By Egorov's Theorem, Shannon-McMillan-Breiman Theorem, and Lemma \ref{lem:projection_lyapunov}, there exist $C \ge 1$, $E_1\subset \Sigma$ with $\nu(E_1)>1-\eps$, and $E_2\subset G(d,d-k)\times\Sigma$ with $\mu^{d-k}_F\times\nu(E_2)>1-\varepsilon$ such that
  \begin{align*}
    C^{-1}e^{-n(h_{\nu}+\varepsilon)} &\leq \nu([\iii|_n]) \leq  Ce^{-n(h_{\nu}-\varepsilon)}, \\
    C^{-1}e^{-n( \chi_1 + \cdots + \chi_{\lfloor s \rfloor} + (s-\lfloor s \rfloor)\chi_{\lceil s \rceil}+\varepsilon)}  &\leq \fii^s(\proj_{V^\perp}A_{\jjj|_n}) \leq  Ce^{-n( \chi_1 + \cdots + \chi_{\lfloor s \rfloor} + (s-\lfloor s \rfloor)\chi_{\lceil s \rceil}-\varepsilon)}
  \end{align*}
  for every $\ii\in E_1$ and $(V,\jjj)\in E_2$. By Fubini's Theorem and Lemma~\ref{lem:transint}, we have
  \begin{align*}
    \iiint \frac{\mathrm{d}\nu|_{E_1}(\iii)\,\mathrm{d}(\mu_F^{d,d-k}\times\nu)|_{E_2}(V,\jjj)\,\mathrm{d}m(u)}{|\proj_{V^{\perp}}\pi_u(\iii)-\proj_{V^{\perp}}\pi_u(\jjj)|^s}
    &= \iiint\frac{\mathrm{d}m(u)\,\mathrm{d}\nu|_{E_1}(\iii)\,\mathrm{d}(\mu_F^{d,d-k}\times\nu)|_{E_2}(V,\jjj)}{|\proj_{V^{\perp}}\pi_u(\iii)-\proj_{V^{\perp}}\pi_u(\jjj)|^s} \\
    & \leq C'\iint\frac{\mathrm{d}(\mu_F^{d,d-k}\times\nu)|_{E_2}(V,\jjj)\,\mathrm{d}\nu|_{E_1}(\iii)}{\fii^s(\proj_{V}A_{\iii\wedge\jjj})}
  \end{align*}
  for some constant $C'>0$. By decomposing the space $\Sigma\times G(d,d-k)$ into $\{\iii\}\times G(d,d-k)$ and $\bigcup_{n=0}^{\infty}\left\{\jjj:|\iii\wedge\jjj|=n\right\}\times G(d,d-k)$, we get
  \begin{align*}
    C'\iint&\frac{\mathrm{d}(\mu_F^{d,d-k}\times\nu)|_{E_2}(V,\jjj)\,\mathrm{d}\nu|_{E_1}(\iii)}{\fii^s(\proj_{V}A_{\iii\wedge\jjj})} \\
    &\leq CC'\sum_{n=0}^{\infty}\int e^{n( \chi_1 + \cdots + \chi_{\lfloor s \rfloor} + (s-\lfloor s \rfloor)\chi_{\lceil s \rceil}+\varepsilon)}(\mu_F^{d,d-k}\times\nu)|_{E_2}(G(d,d-k) \times [\iii|_n])\,\mathrm{d}\nu|_{E_1}(\iii) \\
    &=CC'\sum_{n=0}^{\infty} e^{n( \chi_1 + \cdots + \chi_{\lfloor s \rfloor} + (s-\lfloor s \rfloor)\chi_{\lceil s \rceil}+\varepsilon)}\sum_{|\kkk|=n}(\mu_F^{d,d-k}\times\nu)|_{E_2}(G(d,d-k) \times [\kkk])\nu|_{E_1}([\kkk]) \\
    &\leq C^2C'\sum_{n=0}^{\infty} e^{n( \chi_1 + \cdots + \chi_{\lfloor s \rfloor} + (s-\lfloor s \rfloor)\chi_{\lceil s \rceil}+\varepsilon)}e^{-n(h_{\nu}-\varepsilon)}\sum_{|\kkk|=n}(\mu_F^{d,d-k}\times\nu)|_{E_2}(G(d,d-k) \times [\kkk])\\
    &\leq C^2C'\sum_{n=0}^{\infty} e^{n( \chi_1 + \cdots + \chi_{\lfloor s \rfloor} + (s-\lfloor s \rfloor)\chi_{\lceil s \rceil}+\varepsilon)}e^{-n(h_{\nu}-\varepsilon)}
  \end{align*}
  Since $s\chi_{\lceil s \rceil} < \chi_{\lceil s \rceil}\diml\nu - 2\eps = \lfloor s \rfloor\chi_{\lceil s \rceil} + h_\nu - \sum_{j=1}^{\lfloor s \rfloor} \chi_j - 2\eps$ we have finished the proof.
\end{proof}

We say that an ergodic measure $\nu$ on $\Sigma$ satisfies the \emph{Ledrappier-Young formula} for a tuple $\Av=(A_1,\dots,A_N) \in GL_d(\R)$ of contractive matrices if
\begin{align*}
  \dim \pi_{\Av,\tv}\nu &= \sum_{k=1}^{d-1} \frac{\chi_{k+1}(\Av,\nu)-\chi_{k}(\Av,\nu)}{\chi_d(\Av,\nu)} \dim\proj_{V_k^\perp}\pi_{\Av,\tv}\nu + \frac{h_\nu-H_{\Av,\tv,\nu}}{\chi_d(\Av,\nu)}
\end{align*}
for all $\tv \in (\R^d)^N$ and for $\mu_F^1 \times \cdots \times \mu_F^{d-1}$-almost all $(V_1,\ldots,V_{d-1}) \in G(d,1) \times \cdots \times G(d,d-1)$.
Here $H_{\Av,\tv,\nu}$ denotes the conditional entropy defined in \cite[\S 2]{BaranyKaenmaki}. We omit its definition since $H_{\Av,\tv,\nu} = 0$ in our considerations; see \cite[Corollary 2.8]{BaranyKaenmaki}.

\begin{lemma}\label{lem:safreg}
  If a Bernoulli measure $\nu$ on $\Sigma$ has simple Lyapunov spectrum, then $\nu$ satisfies the Ledrappier-Young formula. Moreover, every ergodic quasi-Bernoulli measure $\nu$ on $\Sigma$ satisfies the Ledrappier-Young formula for every tuple $\Av\in GL_d(\R)^N$ of contractive matrices satisfying the totally dominated splitting condition.
\end{lemma}

\begin{proof}
  The claims follow immediately from \cite[Theorems~2.3 and 2.6]{BaranyKaenmaki}.
\end{proof}

The next theorem shows that, under the strong separation condition, the Ledrappier-Young formula and the modified self-affine transversality condition together guarantee the desired dimension formula.

\begin{theorem}\label{thm:translationtransversal}
  Let $\tv\in (\R^{d})^N$ and $\nu$ be an ergodic measure on $\Sigma$ satisfying the Ledrappier-Young formula for a tuple $\Av \in GL_d(\R)^N$ of contractive matrices. Assume that $(\mathcal U, m)$ satisfies the modified self-affine transversality condition for $\Av$. If $\Phi_{\Av,u(\tv)}$ satisfies the strong separation condition for all $u \in \mathcal U$ or $\diml\nu\le d-1$, then
  \begin{equation*}
    \dim \pi_{\Av,u(\tv)}\nu = \diml\nu
  \end{equation*}
  for $m$-almost all $u \in \mathcal{U}$.
\end{theorem}

\begin{proof}
  By the Ledrappier-Young formula and Proposition \ref{thm:projection}, we have
  \begin{align*}
    \dim \pi_{\Av,u(\tv)}\nu &= \sum_{k=1}^{d-1} \frac{\chi_{k+1}(\Av,\nu)-\chi_{k}(\Av,\nu)}{\chi_d(\Av,\nu)} \dim\proj_{V_k^\perp}\pi_{\Av,u(\tv)}\nu + \frac{h_\nu-H_{\Av,u(\tv),\nu}}{\chi_d(\Av,\nu)} \\
    &= \sum_{k=1}^{d-1} \frac{\chi_{k+1}(\Av,\nu)-\chi_{k}(\Av,\nu)}{\chi_d(\Av,\nu)} \min\{ k,\diml\nu \} + \frac{h_\nu-H_{\Av,u(\tv),\nu}}{\chi_d(\Av,\nu)} \\
    &= \diml\nu - \frac{H_{\Av,u(\tv),\nu}}{\chi_d(\Av,\nu)}
  \end{align*}
  for $\mu_F^1 \times \cdots \times \mu_F^{d-1}$-almost every $(V_1,\ldots,V_{d-1}) \in G(d,1) \times \cdots \times G(d,d-1)$ and for $m$-almost every $u\in\mathcal{U}$. If $\Phi_{\Av,u(\tv)}$ satisfies the strong separation condition, then \cite[Corollary~2.8]{BaranyKaenmaki} implies that $H_{\Av,u(\tv),\nu}=0$ and the claim follows. Furthermore, if $\diml\mu \le d-1$, then, by choosing a typical $V \in G(d,1)$ in the sense of Proposition \ref{thm:projection}, we have $\diml\nu = \dim\proj_{V^\perp}\pi_{\Av,u(\tv)}\nu \le \dim\pi_{\Av,u(\tv)}\nu = \diml\nu - H_{\Av,u(\tv),\nu}/\chi_d(\Av,\nu)$. Therefore $H_{\Av,u(\tv),\nu}=0$ also in this case.
\end{proof}

\begin{remark}\label{ex:nonsharpness}
We indicate that some assumptions on the matrix norms in the statement of Theorem \ref{thm:mainplanar} are necessary, at least for our approach of proof.
As we will see in \S \ref{sec:first-proof}, the strategy of the proof for Theorem \ref{thm:mainplanar} relies heavily on Theorem \ref{thm:translationtransversal}. In order to
apply Theorem \ref{thm:translationtransversal} on a planar self-affine system, we only need to check that the strong separation condition and the modified self-affine
transversality condition hold, because by Lemma \ref{lem:safreg} the Ledrappier-Young formula holds for Bernoulli measures of planar self-affine systems.

We consider the example of Przytycki and Urba\'nski \cite[Theorem~8]{PrzytyckiUrbanski89}. They investigate an IFS $\Phi_{\Av,\tv}$ with $\Av=(A_1,A_2)$, where
$$
A_i=
\begin{pmatrix}
  \lambda & 0 \\
  0 & \gamma \\
\end{pmatrix}
$$
such that $\lambda>1/2>\gamma$ and $\lambda^{-1}$ is a Pisot number, and ${\bf v}=\{(0,0),(1-\lambda,1-\gamma)\}$. They prove that in this case, for the equidistributed
Bernoulli measure $\mu$, $\dimh \pi\mu<\diml\mu$. Notice that here, the strong separation condition holds for $\Fb_{\Av, \tv}$ and by varying $\lambda$ and $\gamma$ we
can break the condition \eqref{eq:matricesplanar}.

We will see in Lemma \ref{lem:trans2} below that for any system satisfying the condition \eqref{eq:matricesplanar} the modified self-affine transversality condition does
hold for $O(d)$ equipped with the Haar measure (recall that $O(d)\subset GL_2(\mathbb R)$ denotes the orthogonal group). However, this is not the case here. For $G\in O(d)$, the
Furstenberg measure for the system $\Phi_{\Av, G(\tv)}$ is the Dirac measure supported on $V=\mathrm{span}\{(0,1)\}$ and $\dimh\proj_{V^{\perp}}\pi_G\mu<1$.
Therefore $\dimh\pi_G\mu<\diml\mu$ for all $G\in O(d)$. Observe that therefore the claim of Theorem \ref{thm:translationtransversal} does not hold, and consequently, the modified self-affine transversality condition does not hold.
\end{remark}

We say that two affine iterated function systems $\Fb$ and $\Psi$ are \emph{equivalent} if the self-affine set of $\Psi$ is an isometric copy of the self-affine set of $\Fb$. This equivalence obviously preserves all the dimensional and separation properties of the self-affine set.

Let $\eta>0$ and $m$ be a probability measure on a group $\mathcal{U}$ contained in $\{ u\in GL_d(\R) : \eta < \alpha_d(u) \le \alpha_1(u) < \eta^{-1} \}$. Notice that all the eigenvalues of the matrices in $\mathcal U$ are one in modulus. We will introduce a method which allows us to handle matrices as parameters when using the modified self-affine transversality condition. If $u\in \mathcal U$ and $\Av = (A_1,\ldots,A_N) \in GL_d(\R)^N$, then we define $u(\Av) = (u^{-1}A_1u, \dots, u^{-1}A_Nu)$. Recall that $u(\tv) = (uv_1, \dots, uv_N)$ for $\tv = (v_1,\ldots,v_N) \in (\R^d)^N$.

\begin{lemma}\label{lem:equivalence}
  If $\tv \in (\R^d)^N$ and $\Av \in GL_d(\R)^N$, then the affine iterated function systems $\Fb_{u(\Av),\tv}$ and $\Fb_{\Av,u(\tv)}$ are equivalent for all $u \in \mathcal U$.
\end{lemma}

\begin{proof}
Recall that $f(A,v)$ is the affine mapping $x \mapsto Ax+v$. Let $\Av = (A_1,\dots,A_N) \in GL_d(\R)^N$ and $\tv = (v_1,\dots, v_N) \in (\R^d)^N$. Clearly, for each $x\in \R^d$ we have
\[
f(u^{-1}A_iu,v_i)(x)= f(u^{-1}, 0) ( f(A_i, u(v_i)) (f(u, 0)(x)))
\]
and $f(u^{-1},0) (f(u, 0) (x)) = x$. Thus, $\Fb_{u(\Av),\tv}$ and $\Fb_{\Av,u(\tv)}$ are equivalent.
\end{proof}

We define a partition of $GL_d(\R)^N$ by setting an equivalence relation on $GL_d(\R)^N$ as follows: if $\Av, {\bf B}\in GL_d(\R)^N$, then we say that $\Av\sim{\bf B}$ if and only if there exists $u\in \mathcal U$ such that ${\bf B} =u(\Av)$. Since $\mathcal{U}$ is a group $\sim$ defines an equivalence relation. Hence
$$
  \mathcal{P}(\Av)=\{{\bf B}\in GL_d(\R)^N : \Av\sim{\bf B}\}.
$$
is a partition of $GL_d(\R)^N$.
Since $GL_d(\R)^N$ equipped with the distance $d(\Av,{\bf B})=\max_i\|A_i-B_i\|$ is a separable metric space (matrix tuples with rational entries form a countable dense subset), the $\sigma$-algebra $\mathcal{B}$ of Borel sets of $GL_d(\R)^N$ is countably generated by $\{X_1,X_2,\dots\}$, where $X_i$ is open in $GL_d(\R)^N$. By defining $X'_i=\bigcup_{\Av\in X_i}\mathcal{P}(\Av)$, the set $\{X_1',X_2',\dots\}$ generates the $\sigma$-algebra $\mathcal{B}_{\mathcal{P}}=\left\{X\in\mathcal{B} : X=\bigcup_{\Av\in X}\mathcal{P}(\Av)\right\}$. Indeed, we clearly have $\sigma(\{X_1',X_2',\dots\})\subset\mathcal{B}_{\mathcal{P}}$. On the other hand, if $X\in\mathcal{B}_{\mathcal{P}}$ is open, then $X=\bigcup_j X_{i_j}$. But now $X=\bigcup_{\Av\in X}\mathcal{P}(\Av)=\bigcup_j\bigcup_{\Av\in X_{i_j}}\mathcal{P}(\Av)=\bigcup_j X_j'$ and thus $\sigma(\{X_1',X_2',\dots\})\supset\mathcal{B}_{\mathcal{P}}$.

By Rokhlin's disintegration theorem (see \cite{Rohlin52} and \cite{Simmons2012}), for any finite measure $\mathcal{M}$ on $GL_d(\R)^N$, there exists a family of measures $\{\mathcal{M}^{\mathcal{P}(\Av)}\}$ such that $\mathcal{M}^{\mathcal{P}(\Av)}$ is uniquely defined for $\mathcal{M}$-almost every $\Av$, $\mathcal{M}^{\mathcal{P}(\Av)}$ is supported on $\mathcal{P}(\Av)$, and, moreover,
$$
\mathcal{M}=\int\mathcal{M}^{\mathcal{P}(\Av)}\,\mathrm{d}\mathcal{M}(\Av).
$$
The combination of Lemma~\ref{lem:equivalence} and the definition of $\mathcal{P}$ implies that we have the same Furstenberg measures for all $u$ on the partition elements. More precisely, for each ${\bf B}\in\mathcal{P}(\Av)$ there is $u\in\mathcal{U}$ such that $\Phi_{{\bf B},{\bf v}}$ is equivalent to $\Phi_{{\bf A},u({\bf v})}$ whose Furstenberg measures are independent of $u\in\mathcal{U}$.

\begin{theorem}\label{thm:aeformeasure}
  Let $\tv \in (\R^d)^N$, $\eta>0$, and $m$ be a probability measure on a group $\mathcal{U}$ contained in $\{ u\in GL_d(\R) : \eta < \alpha_d(u) \le \alpha_1(u) < \eta^{-1} \}$. Furthermore, let $\mathcal{M}$ be a probability measure on $GL_d(\R)^N$ and let $\xi_{\Av} \colon \mathcal{U} \to GL_d(\R)^N$, $\xi_{\Av}(u)=u(\Av)$, be such that $\mathcal{M}^{\mathcal{P}(\Av)}=\xi_\Av m$ for $\mathcal{M}$-almost all $\Av$. Assume that $\nu$ is an ergodic measure on $\Sigma$ satisfying the Ledrappier-Young formula and $(\mathcal U, m)$ satisfies the modified self-affine transversality condition for $\mathcal M$-almost all tuples $\Av$ of contractive matrices. Then
  $$
    \dimh \pi_{\Av,\tv}\nu=\diml\nu
  $$
  for $\mathcal{M}$-almost all tuples $\Av\in GL_d(\R)^N$ of contractive matrices for which $\Fb_{\Av, \tv}$ satisfies the strong separation condition or $\diml \nu\leq d-1$.
\end{theorem}

\begin{proof}
  Observe that, by Theorem~\ref{thm:translationtransversal}, if $\Phi_{\Av,u(\tv)}$ satisfies the strong separation condition for all $u \in \mathcal U$ and $\mathcal M$-almost all tuples $\Av$ of contractive matrices or $\diml\nu\le d-1$, then
  $$
    \dim\pi_{\Av,u(\tv)}\nu=\diml\nu
  $$
  for $m$-almost all $u \in \mathcal U$ and $\mathcal M$-almost all tuples $\Av$ of contractive matrices.
  By Lemma \ref{lem:equivalence}, we have $\dimh \pi_{u(\Av),\tv}\nu = \dimh \pi_{\Av,u(\tv)}\nu$ for all $u \in \mathcal U$. Thus, by Rokhlin's disintegration theorem,
  $$
    \mathcal{M}=\int\xi_{\Av}m\,\mathrm{d}\mathcal{M}(\Av),
  $$
  and hence, $\dimh\pi_{\Av,\tv}\nu=\diml\nu$ for $\mathcal{M}$-almost all tuples $\Av$ of contractive matrices.
\end{proof}

\section{Dimension via sub-systems} \label{sec:sub-systems}

Let $\mathbb{A}$ be the collection of all tuples $\Av \in GL_d(\R)^N$ of contractive matrices that satisfy $\chi_i(\Av,\mu) \ne \chi_j(\Av,\mu)$ for $i \ne j$ where $\mu$ is an ergodic $s$-equilibrium state of $\Av$ and $s=\dimaff(\Av)$.
We note that the functions $\Av\mapsto\chi_i(\Av,\nu)$ and $\nu\mapsto\chi_i(\Av,\nu)$ are lower semi-continuous in the usual and weak*-topology. Thus, by \cite[Propositions 3.4 and 3.6]{KaenmakiBing} and \cite[Proposition~5.2]{FengShmerkin14}, the set $\mathbb{A}$ is Borel measurable in $GL_d(\R)^N$.

In this section, we show that equilibrium states can be approximated by step-$n$ Bernoulli measure arbitrarily well on $\mathbb{A}$. Each affine iterated function system contains a well-approximating sub-system in which the Ledrappier-Young formula holds. In our setting, this observation yields a dimension formula for the self-affine set. Motivation for this section is purely technical: at the moment, we do not know whether ergodic equilibrium states satisfy the Ledrappier-Young formula -- Lemma \ref{lem:safreg} gathers the current knowledge on this problem.

\begin{proposition}\label{prop:FengShmerkin}
	For every $\Av\in\mathbb{A}$ and $\eps>0$ there exist $n \in \N$ and a step-$n$ Bernoulli measure $\nu$ such that $\diml\nu \ge \dimaff\Av - \eps$.
\end{proposition}

\begin{proof}
  Fix $\Av = (A_1,\ldots,A_N) \in \mathbb{A}$, define $s =\dimaff\Av$, and let $\mu$ be an ergodic $s$-equilibrium state of $\Av$. By Feng and Shmerkin \cite[Theorem 3.3]{FengShmerkin14}, there exist $\eta>0$ and an infinite set $S \subset \N$ such that for every $n \in S$ there is $\Gamma_n \subset \Sigma_n$ with $\sum_{\iii \in \Gamma_n} \mu([\iii]) \ge \eta$. Moreover, since all the Lyapunov exponents are distinct it also follows that each $(A_\iii)_{\iii \in \Gamma_n}$ satisfies the totally dominated splitting condition and hence, via e.g.\ \cite[Lemma 2.1]{FengShmerkin14}, there exists a constant $C \ge 1$ such that
  \begin{equation} \label{eq:cone-condition}
  \fii^s(A_\iii)\fii^s(A_\jjj) \le C\fii^s(A_{\iii\jjj})
  \end{equation}
  for all $\iii, \jjj \in \bigcup_{k=1}^\infty \Gamma_n^k$, $k \in \N$, $n \in S$, and $s \ge 0$. For each $n \in S$, let us choose $s_n$ such that
  \begin{equation*}
          \sum_{\iii \in \Gamma_n} \fii^{s_n}(A_\iii) = 1.
  \end{equation*}
  We will show that $s_n \to s$ as $n \to \infty$.

  Applying Egorov's Theorem, it follows from \eqref{eq:eq-state} and Shannon-McMillan-Breiman Theorem that there exists a set $E \subset \Sigma$ with $\sum_{\iii \in \Gamma_n} \mu([\iii] \cap E) \ge \eta/2$ such that for every $\iii \in \bigcup_{\jjj \in \Gamma_n} [\jjj] \cap E$
  \begin{equation*}
          \lim_{n \to \infty} \tfrac{1}{n}\log\frac{\mu([\iii|_{n}])}{\fii^{s}(A_{\iii|_{n}})} = 0
  \end{equation*}
  uniformly. This implies that there exists a sequence $(c_n)_{n \in \N}$ of reals such that $\lim_{n \to \infty} \log c_n^{1/n} = 0$ and
  \begin{equation*}
          c_n^{-1}\fii^{s}(A_\iii) \le \mu([\iii]) \le c_n\fii^{s}(A_\iii)
  \end{equation*}
  for all $\iii \in \bigcup_{\jjj \in \Gamma_n} [\jjj] \cap E$ and $n \in S$. Therefore, by \eqref{eq:svf2}, we have
  \begin{equation*}
          c_n^{-1}\eta/2 \le c_n^{-1}\sum_{\iii \in \Gamma_n} \mu([\iii] \cap E) \le \sum_{\iii \in \Gamma_n} \fii^{s}(A_\iii) \le \|\Av\|^{n(s-s_n)} \sum_{\iii \in \Gamma_n} \fii^{s_n}(A_\iii) = \|\Av\|^{n(s-s_n)}
  \end{equation*}
  and thus,
  \begin{equation} \label{eq:sn-lower}
  s-s_n \le \frac{\log c_n^{-1}\eta/2}{n\log\|\Av\|}
  \end{equation}
  for all $n \in S$. On the other hand, again by \eqref{eq:svf2}, we have
  \begin{equation*}
          \sum_{\iii \in \Sigma_n} \fii^{s}(A_\iii) \ge \sum_{\iii \in \Gamma_n} \fii^{s}(A_\iii) \ge \mathfrak{m}(\Av)^{n(s-s_n)} \sum_{\iii \in \Gamma_n} \fii^{s_n}(A_\iii) = \mathfrak{m}(\Av)^{n(s-s_n)}
  \end{equation*}
  and hence,
  \begin{equation} \label{eq:sn-upper}
  s-s_n \ge \frac{\log\sum_{\iii \in \Sigma_n} \fii^{s}(A_\iii)}{n\log\mathfrak{m}(\Av)}
  \end{equation}
  for all $n \in S$. By the definition of the affinity dimension, \eqref{eq:sn-lower} and \eqref{eq:sn-upper} show that $s_n \to s$ as $n \to \infty$.

  Let $\nu_{\pv_n}$ be the Bernoulli measure obtained from the probability vector $\pv_n = (\fii^{s_n}(A_\iii))_{\iii\in\Gamma_n}$. Observe that, by \eqref{eq:cone-condition},
  \begin{align*}
          h_{\pv_n} &= -\sum_{\iii \in \Gamma_n} \nu_{\pv_n}([\iii])\log\fii^{s_n}(A_\iii) \\
          &= -\tfrac{1}{k}\sum_{\iii_1,\ldots,\iii_k \in \Gamma_n} \nu_{\pv_n}([\iii_1]) \cdots \nu_{\pv_n}([\iii_k]) \log\fii^{s_n}(A_{\iii_1}) \cdots \fii^{s_n}(A_{\iii_k}) \\
          & \ge -\tfrac{1}{k}\sum_{\iii_1,\ldots,\iii_k \in \Gamma_n} \nu_{\pv_n}([\iii_1]) \cdots \nu_{\pv_n}([\iii_k]) \log C^{k-1}\fii^{s_n}(A_{\iii_1\cdots\iii_k})
  \end{align*}
  and similarly,
  \begin{align*}
          h_{\pv_n} &= -\sum_{\iii \in \Gamma_n} \nu_{\pv_n}([\iii])\log\fii^{s_n}(A_\iii) \\
          & \le -\tfrac{1}{k}\sum_{\iii_1,\ldots,\iii_k \in \Gamma_n} \nu_{\pv_n}([\iii_1]) \cdots \nu_{\pv_n}([\iii_k]) \log \fii^{s_n}(A_{\iii_1\cdots\iii_k}).
  \end{align*}
  Letting $k \to \infty$ this implies
  \begin{align*}
          \diml\nu_{\pv_n}=\lfloor s_n\rfloor + \frac{h_{\pv_n}-\sum_{j=1}^{\lfloor s_n\rfloor} \chi_j(\Av^n,\pv_n)}{\chi_{\lceil s_n\rceil}(\Av^n,\pv_n)}  &\ge \frac{-\log C}{\chi_{\lceil s_n\rceil}(\Av^n,\pv_n)} + s_n \ge \frac{\log C}{n\log \mathfrak{m}(\Av)} + s_n,
  \end{align*}
  where $\Av^n = (A_\iii)_{|\iii| = n}$. This is what we wanted to show.
\end{proof}

We will next transfer the previous proposition in the form we can use in our setting.

\begin{theorem}\label{thm:almosteverymatrix}
  Let $\tv \in (\R^d)^N$, $\eta>0$, and $m$ be a probability measure on a group $\mathcal{U}$ contained in $\{ u\in GL_d(\R) : \eta < \alpha_d(u) \le \alpha_1(u) < \eta^{-1} \}$. Furthermore, let $\mathcal{M}$ be a probability measure on $GL_d(\R)^N$ such that $\mathcal M(\{ \Av \in GL_d(\R)^N : \|\Av\|<1 \} \setminus \mathbb A) = 0$ and let $\xi_{\Av} \colon \mathcal{U} \to GL_d(\R)^N$, $\xi_{\Av}(u)=u(\Av)$, be such that $\mathcal{M}^{\mathcal{P}(\Av)}=\xi_\Av m$ for $\mathcal{M}$-almost all $\Av$. Assume that $(\mathcal U, m)$ satisfies the modified self-affine transversality condition for $\mathcal M$-almost all tuples $\Av$ of contractive matrices. Then
  $$
    \dimh E_{\Av,\tv}=\udimm E_{\Av,\tv}=\dimaff\Av
  $$
  for $\mathcal{M}$-almost all tuples $\Av\in GL_d(\R)^N$ of contractive matrices for which $\Fb_{\Av, \tv}$ satisfies the strong separation condition or $\dimaff\Av \leq d-1$.
\end{theorem}

\begin{proof}
  Let us first show that for every $\Av\in\mathbb{A}$ for which $\Fb_{\Av, \tv}$ satisfies the strong separation condition or $\dimaff \Av\leq d-1$, and for every $\varepsilon>0$ it holds that
  \begin{equation}\label{eq:needtoprove}
    \dimh E_{\Av,u(\tv)} > \dimaff\Av-\varepsilon
  \end{equation}
  for $m$-almost every $u\in\mathcal U$. Fix such a tuple $\Av\in\mathbb{A}$ and $\eps>0$. Then, as in the proof of Proposition~\ref{prop:FengShmerkin}, we find a finite set $\Gamma\subset\Sigma_*$ and a Bernoulli measure $\nu$ on $\Gamma^{\mathbb{N}}$ such that $(A_{\iii})_{\iii\in\Gamma}$ satisfies the totally dominated splitting condition and $\diml\nu\geq \dimaff \Av-\varepsilon$. By Lemma~\ref{lem:safreg}, $\nu$ satisfies the Ledrappier-Young formula. Hence, by Theorem~\ref{thm:translationtransversal}, $\dim\pi_{\Av,u(\tv)}\nu=\diml\nu$ for $m$-almost every $u \in \mathcal U$ and we have finished the proof of \eqref{eq:needtoprove}.

  By Lemma \ref{lem:equivalence}, we have $\dimh E_{\Av,u(\tv)} = \dimh E_{u(\Av),\tv}$ for all $u \in \mathcal U$. It is also easy to see that $\dimaff \Av = \dimaff u(\Av)$ for every $u\in\mathcal{U}$. Since $\mathbb{A}$ has full measure with respect to $\mathcal{M}$, by Rokhlin's disintegration theorem, we thus have $\dimh E_{\Av,\tv} > \dimaff\Av-\varepsilon$ for $\mathcal{M}$-almost all tuples $\Av$ of contractive matrices for which $\Fb_{\Av, \tv}$ satisfies the strong separation condition or $\dimaff\Av \leq d-1$. Since $\varepsilon>0$ was arbitrary, the proof is complete.
\end{proof}

\section{The planar case} \label{sec:first-proof}

In this section, we prove Theorem \ref{thm:mainplanar} as an application of Theorems~\ref{thm:aeformeasure} and \ref{thm:almosteverymatrix} by showing that the modified self-affine transversality condition holds for some pair $(\mathcal U, m)$ for all $\Av\in \mathcal A_\tv$. Recall that if $\tv=(v_1,\dots,v_N)\in(\R^2)^N$ is such that $v_i\neq v_j$ for every $i\neq j$, then
$$
  \mathcal{A}_{\tv}=\biggl\{\Av\in GL_2(\R)^N : 0<\max_{i\neq j}\frac{\|A_i\|+\|A_j\|}{|v_i-v_j|}\cdot\frac{\|\tv\|}{1-\|\Av\|}<\frac{\sqrt{2}}{2}\biggr\}.
$$
Let $SO(2)$ be the special orthogonal group of $GL_2(\R)$. Note that $SO(2) = \{u_{\alpha}\}_{\alpha \in \R}$, where
\[
  u_\alpha =
  \begin{pmatrix}
    \cos(\alpha) & -\sin(\alpha) \\
    \sin(\alpha) & \cos(\alpha)
  \end{pmatrix}
\]
is a rotation by an angle $\alpha\in \R$. To simplify notation, we will denote $\pi_{\Av,u_\alpha(\tv)}$ by $\pi_{\alpha}$. Furthermore, we denote the differential with respect to $\alpha$, evaluated at $\alpha_0$, by $\partial_{\alpha=\alpha_0}$. The following is a transversality lemma suitable for our purposes. 

\begin{lemma}\label{lem:trans}
  Let $\tv = (v_1,\dots, v_N) \in (\R^2)^N$ be such that $v_i\neq v_j$ for $i\neq j$ and $\Av = (A_1,\dots,A_N) \in \mathcal{A}_{\tv}$. Then there exists $\delta>0$ such that for every $\alpha_0\in[0,2\pi]$, for every $v\in\R^2$ with $|v|=1$, and for every $\bi,\bj\in\Sigma$ with $\iii|_1\neq \jjj|_1$
  $$
    |\inner{v}{\pi_{\alpha_0}(\ii)-\pi_{\alpha_0}(\bj)}|\ge\delta\quad\text{or}\quad|\partial_{\alpha=\alpha_0}\inner{v}{\pi_{\alpha}(\ii)-\pi_{\alpha}(\bj)}|\ge\delta.
  $$
\end{lemma}

\begin{proof}
  Let us argue by contradiction. Suppose that for every $\delta>0$ there exists $\alpha_0\in[0,2\pi]$, $v\in\R^2$ with $|v|=1$ and $\bi,\bj \in \Sigma$ with $i_1\neq j_1$ such that
  $$
    |\inner{v}{\pi_{\alpha_0}(\ii)-\pi_{\alpha_0}(\jj)}|<\delta \quad \text{and} \quad |\partial_{\alpha=\alpha_0} \inner{v}{\pi_{\alpha}(\ii)-\pi_{\alpha}(\bj)}|<\delta.
  $$
  Let $(\delta_n)_{n \in \N}$ be a sequence of positive reals such that $\delta_n\to0$ as $n\to\infty$. By compactness, there exist $\alpha_0\in[0,2\pi]$, $v\in\R^2$ with $|v|=1$ and  $\ii,\bj\in\Sigma$ with $\bi|_1\neq \bj|_1$ such that
  $$
    |\inner{v}{\pi_{\alpha_0}(\ii)-\pi_{\alpha_0}(\bj)}|=0 \quad \text{and} \quad |\partial_{\alpha=\alpha_0}\inner{v}{\pi_{\alpha}(\ii)-\pi_{\alpha}(\bj)}|=0.
  $$
  Since $\partial_\alpha u_\alpha^{-1} = u_{\alpha+\pi/2}^{-1}$ we have $\partial_\alpha\pi_{\alpha}(\ii)=\pi_{\alpha+\pi/2}(\ii)$ for every $\ii\in\Sigma$. Thus,
  \begin{align*}
    0&=\inner{v}{\pi_{\alpha_0}(\ii)-\pi_{\alpha_0}(\bj)} \\
    &=\inner{u_{\alpha_0}v}{v_{i_1}-v_{j_1}}+\inner{v}{A_{i_1}\pi_{\alpha_0}(\sigma(\ii))-A_{j_1}\pi_{\alpha_0}(\sigma(\bj))}
  \end{align*}
  and
  \begin{align*}
    0&=\partial_{\alpha=\alpha_0}\inner{v}{\pi_{\alpha}(\ii)-\pi_{\alpha}(\bj)}\\
    &=\inner{u_{\alpha_0+\pi/2}v}{v_{i_1}-v_{j_1}}+\inner{v}{A_{i_1}\pi_{\alpha_0+\pi/2}(\sigma(\ii))-A_{j_1}\pi_{\alpha_0+\pi/2}(\sigma(\bj))}.
  \end{align*}
  Hence,
  \begin{align*}
    |v_{i_1}-v_{j_1}|^2
    &=\bigl(\inner{v}{A_{i_1}\pi_{\alpha_0}(\sigma(\ii))-A_{j_1}\pi_{\alpha_0}(\sigma(\bj))}\bigr)^2+\bigl(\inner{v}{A_{i_1}\pi_{\alpha_0+\pi/2}(\sigma(\ii))-A_{j_1}\pi_{\alpha_0+\pi/2}(\sigma(\bj))}\bigr)^2\\
    &\leq
    \bigl(|A_{i_1}\pi_{\alpha_0}(\sigma(\ii))|+|A_{j_1}\pi_{\alpha_0}(\sigma(\bj))|\bigr)^2+\bigl(|A_{i_1}\pi_{\alpha_0+\pi/2}(\sigma(\ii))|+|A_{j_1}\pi_{\alpha_0+\pi/2}(\sigma(\bj))|\bigr)^2\\
    &\leq
    2\bigl(\|A_{i_1}\|+\|A_{j_1}\|\bigr)^2\biggl(\frac{\|\tv\|}{1-\|\Av\|}\biggr)^2.
  \end{align*}
  This contradicts the condition \eqref{eq:matricesplanar}.
\end{proof}

The following proposition shows that $(SO(2),\LL^1)$ satisfies the modified self-affine transversality condition for all $\Av \in \mathcal{A}_{\tv}$.

\begin{proposition} \label{lem:trans2}
  Let $\tv = (v_1,\dots, v_N) \in (\R^2)^N$ be such that $v_i\neq v_j$ for $i\neq j$ and $\Av = (A_1,\dots,A_N) \in \mathcal{A}_{\tv}$. Then there is a constant $C'>0$ such that
  $$
    \mathcal{L}^1\left(\left\{\alpha\in[0,2\pi]: |\proj_{\theta}\pi_{\alpha}(\ii)-\proj_{\theta}\pi_{\alpha}(\jj) |<t\right\}\right) \leq C'\min\biggl\{1,\frac{t}{\alpha_1(\proj_{\theta}A_{\ii\wedge\jj})}\biggr\}.
  $$
  for all $\ii,\bj\in\Sigma$ with $\iii\neq \jjj$, $t>0$, and $\theta\in G(2,1)$.
\end{proposition}

\begin{proof}
  Fix $\iii,\jjj \in \Sigma$ with $\iii \ne \jjj$. Observe that for every $\theta\in G(2,1)$ and $w\in\R^2$ we have
  $$
    |\proj_{\theta}w |=|\inner{v}{w}|,
  $$
  where $v\in\theta$ is so that $|v|=1$. Thus, writing $n=|\ii\wedge\jj|$ we have
  \begin{align*}
  |\proj_{\theta}\pi_{\alpha}(\ii)-\proj_{\theta}\pi_{\alpha}(\jj)| & =  |\inner{v}{\pi_{\alpha}(\ii)-\pi_{\alpha}(\jj)}|
  =|\inner{v}{A_{\ii\wedge\jj}\pi_{\alpha}(\sigma^n\ii)-A_{\ii\wedge\jj}\pi_{\alpha}(\sigma^n\jj)}|\\
  &=|A^T_{\ii\wedge\jj}v|\biggl|\biggl\la\frac{A_{\ii\wedge\jj}^Tv}{|A^T_{\ii\wedge\jj}v|},\pi_{\alpha}(\sigma^n\ii)-\pi_{\alpha}(\sigma^n\jj)\biggl\ra\biggr|.
  \end{align*}
  On the other hand,
  $$
  \|\proj_{\theta}A_{\ii\wedge\jj}\|=\|A_{\ii\wedge\jj}^T\proj_{\theta}^T\|=\sup_{w\in\R^2}\frac{|A_{\ii\wedge\jj}^T\proj_{\theta}^Tw|}{|w|}=\sup_{w\in\R^2}\frac{|A_{\ii\wedge\jj}^T\proj_{\theta}w|}{|\proj_{\theta}w|}\frac{|\proj_{\theta}w|}{|w|}=\|A^T_{\ii\wedge\jj}|\theta\|.
  $$
  Hence, it suffices to prove that for any $v\in\R^2$ with $|v|=1$
  \begin{equation}\label{eq:enough}
    \mathcal{L}^1\left(\left\{\alpha\in[0,2\pi]:|\inner{v}{\pi_{\alpha}(\ii)-\pi_{\alpha}(\jj)}|<t\right\}\right)\leq C't
  \end{equation}
  for all $\iii,\jjj \in \Sigma$ with $\iii|_1 \ne \jjj|_1$.

  Fix $v \in \R^2$ with $|v|=1$ and $\iii,\jjj \in \Sigma$ with $\iii|_1 \ne \jjj|_1$. Write $$H_t = \left\{\alpha\in[0,2\pi]:|\inner{v}{\pi_{\alpha}(\ii)-\pi_{\alpha}(\jj)}|<t\right\}$$ and let $\delta > 0$ be as in Lemma \ref{lem:trans}. Define
  \begin{align*}
    I &= \{ \alpha_0 \in [0,2\pi] : |\inner{v}{\pi_{\alpha_0}(\iii) - \pi_{\alpha_0}(\jjj)}| \ge \delta \}, \\
    J &= \{ \alpha_0 \in [0,2\pi] : |\partial_{\alpha=\alpha_0}\inner{v}{\pi_{\alpha}(\iii) - \pi_{\alpha}(\jjj)}| \ge \delta \}.
  \end{align*}
  Lemma \ref{lem:trans} guarantees that $I \cup J = [0,2\pi]$.
  We trivially have
  \begin{equation*}
    \mathcal{L}^1(H_t\cap I) \leq
    \begin{cases}
      2\pi\leq 2\pi\delta^{-1}t, &\text{if }t\geq\delta, \\
      0, &\text{if }t<\delta.
    \end{cases}
  \end{equation*}
  It remains to estimate the integral in the complement of $I$. Since $[0,2\pi] \setminus I \subset J$ is open, it can be written in a unique way as a union of disjoint open intervals $I_j$ in which the function $\alpha \mapsto \inner{v}{\pi_{\alpha}(\iii) - \pi_{\alpha}(\jjj)}$ is strictly monotone. If $\beta_1$ and $\beta_2$ are the end points of $I_j$, then $\inner{v}{\pi_{\beta_1}(\iii)-\pi_{\beta_1}(\jjj)} = \delta$ and $\inner{v}{\pi_{\beta_2}(\iii)-\pi_{\beta_2}(\jjj)} = -\delta$ or vice versa. Since
  \begin{align*}
    \partial_{\alpha=\alpha_0} \inner{v}{\pi_{\alpha}(\iii)-\pi_{\alpha}(\jjj)} &= \inner{v}{\pi_{\alpha_0+\pi/2}(\iii)-\pi_{\alpha_0+\pi/2}(\jjj)} \\
    &\le |\pi_{\alpha_0+\pi/2}(\iii) - \pi_{\alpha_0+\pi/2}(\jjj)| \le \frac{2\|\tv\|}{1-\|\Av\|}
  \end{align*}
  for all $\alpha_0 \in [0,2\pi]$, we have
  \begin{align*}
    2\delta &= |\inner{v}{\pi_{\beta_2}(\iii)-\pi_{\beta_2}(\jjj)} - \inner{v}{\pi_{\beta_1}(\iii)-\pi_{\beta_1}(\jjj)}| \\
    &= \Bigl| \int_{\beta_1}^{\beta_2} \partial_{\alpha=\alpha_0} \inner{v}{\pi_{\alpha}(\iii)-\pi_{\alpha}(\jjj)} \,\mathrm{d}\LL^1(\alpha_0) \Bigr| \le \frac{2\|\tv\|}{1-\|\Av\|} \LL^1(I_j).
  \end{align*}
  Thus there are only finitely many intervals $I_j$. For each $j$, by the monotonicity and the mean value theorem, there exists $\lambda_j \in I_j$ such that
  \begin{equation*}
    |\inner{v}{\pi_{\alpha}(\iii) - \pi_{\alpha}(\jjj)}| = |\inner{v}{\pi_{\alpha}(\iii) - \pi_{\alpha}(\jjj)} - \inner{v}{\pi_{\lambda_j}(\iii) - \pi_{\lambda_j}(\jjj)}| \ge \delta|\alpha-\lambda_j|
  \end{equation*}
  for all $\alpha \in I_j$. By this estimate, we have
  \begin{align*}
   \mathcal{L}^1(H_t\cap I_j)\leq\mathcal{L}^1(B(\lambda_j,\delta^{-1}t))\leq 2\pi\delta^{-1}t.
  \end{align*}
  Since
  $$
    \mathcal{L}^1(H_t)\leq \mathcal{L}^1(H_t\cap I)+\sum_j\mathcal{L}^1(H_t\cap I_j)\leq C' t,
  $$
  we have shown \eqref{eq:enough} and therefore, finished the proof.
\end{proof}

\begin{proof}[Proof of Theorem \ref{thm:mainplanar}]
  Let $\mathcal{N}$ be the product of Haar measures on $GL_2(\R)^N$. This means that
  $$
    \mathcal{N}(\mathcal B) = \int_{\mathcal B} \prod_{i=1}^N \biggl( \frac{1}{\det A_i} \biggr)^2 \dd\LL^{4N}(\Av).
  $$
  for all measurable $\mathcal B \subset GL_d(\R)^N$.
  Since $\mathcal{L}^{4N}\ll\mathcal{N}$ it suffices to check that the assumptions of Theorems~\ref{thm:aeformeasure} and \ref{thm:almosteverymatrix} hold for the measure $\mathcal{N}$. By the Haar property, for any measurable $L^1(\mathcal{N})$-function $f\colon GL_2(\R)^N\to\R$ and for every $\alpha\in [0,2\pi]$, we have
  \begin{align*}
    \int f(\Av)\,\mathrm{d}\mathcal{N}(\Av) & = \idotsint f(A_1,\dots,A_N)\prod_{i=1}^N\biggl(\frac{1}{\det A_i}\biggr)^2\,\mathrm{d}\mathcal{L}^{4}(A_1)\cdots\mathrm{d}\mathcal{L}^{4}(A_N) \\
    &= \idotsint f(u_{\alpha}^T A_1u_{\alpha},\dots,u_{\alpha}^TA_Nu_{\alpha})\prod_{i=1}^N\biggl(\frac{1}{\det A_i}\biggr)^2\,\mathrm{d}\mathcal{L}^{4}(A_1)\cdots\mathrm{d}\mathcal{L}^{4}(A_N).
  \end{align*}
  Thus, $\mathcal{N}=\int H_{\Av}\mathcal{L}^1\,\mathrm{d}\mathcal{N}(\Av)$, where $H_{\Av}(\alpha)=(u_{\alpha}^T A_1u_{\alpha},\ldots,u_{\alpha}^TA_Nu_{\alpha})$. Therefore, by Proposition~\ref{lem:trans2} and Lemma~\ref{lem:safreg}, the statement for Bernoulli measures in Theorem~\ref{thm:mainplanar} follows from Theorem~\ref{thm:aeformeasure}.

  If $\chi_1(\Av,\mu)=\chi_2(\Av,\mu)$, where $\mu$ is an ergodic equilibrium state of $\Av$, then by Lemma~\ref{lem:ssc-d} and \cite[Theorem~2.13 and Corollary~4.16]{FengHu09} we have $\dimh E_{\Av,\tv}=\dim_{\mathrm{aff}}\Av$. Thus, without loss of generality, we may restrict $\mathcal{N}$ to the set $\mathbb{A}=\left\{\Av:\chi_1(\Av,\mu)>\chi_2(\Av,\mu)\right\}$. Hence, the statement for self-affine sets in Theorem~\ref{thm:mainplanar} follows from Lemma~\ref{lem:ssc-d}, Proposition~\ref{lem:trans2}, and Theorem~\ref{thm:almosteverymatrix} with the measure $\mathcal{M}=\left.\mathcal{N}\right|_{\mathbb{A}}$.
\end{proof}

\section{The higher dimensional case}\label{sec:higherdimcase}

In this section, we prove Theorem~\ref{thm:main} as a consequence of Theorem~\ref{thm:aeformeasure}. At first, we show that the modified self-affine transversality condition holds for some pair $(\mathcal U, m)$ for all $\Av\in \mathcal A_\tv'$. Secondly, we show that Lebesgue typical matrix tuples have simple Lyapunov spectra for all Bernoulli measures, and observe that $\mathcal{D}\subset\mathbb{A}$, where $\mathcal{D}$ is defined in \eqref{eq:d} and $\mathbb{A}$ is as in \S \ref{sec:sub-systems}.

Recall that if $\tv=(v_1,\dots,v_N)\in(\R^d)^N$ is such that $v_i\neq v_j$ for every $i\neq j$, then
\begin{equation*}
  \mathcal{A}_{\tv}'=\biggl\{\Av\in GL_d(\R)^N : 0<\max_{i\neq j}\frac{\|A_i\|+\|A_j\|}{|v_i-v_j|}\cdot\frac{\|\tv\|}{1-\|\Av\|}<\frac{2}{\sqrt{3}}-1\biggr\}.
\end{equation*}
Recall also that
\begin{align*}
  O(d) &= \{ G \in GL_d(\R) : G^TG = GG^T = I \}, \\
  SO(d) &= \{ G \in O(d) : \det(G) = 1 \}
\end{align*}
are the \emph{orthogonal group} and the \emph{special orthogonal group}, respectively. We define $\mathcal U$ to be $O(d)$ and choose $m$ to be the Haar measure $\Theta$ on $O(d)$. The following proposition shows that $(O(d),\Theta)$ satisfies the modified self-affine transversality condition for all $\Av \in \mathcal{A}_{\tv}'$.

\begin{proposition} \label{thm:higher-dim-trans}
  Let $d\in \N$ be such that $d \ge 3$, $\tv = (v_1,\dots, v_N) \in (\R^2)^N$ be such that $v_i\neq v_j$ for $i\neq j$, and $\Av = (A_1,\dots,A_N) \in \mathcal{A}_{\tv}'$. Then there is a constant $C>0$ such that
  $$
    \Theta(\{ G \in O(d): |\proj_V(\pi_{\Av,G(\tv)}(\iii)-\pi_{\Av,G(\tv)}(\jjj))| < t \}) \le C\prod_{i=1}^{\dim V} \min\biggl\{ 1,\frac{t}{\alpha_i(\proj_VA_{\iii\wedge\jjj})} \biggr\}.
  $$
  for all $\ii,\bj\in\Sigma$ with $\iii \neq \jjj$, $t>0$, and proper subspaces $V$ of $\R^d$.
\end{proposition}

\begin{proof}
  Fix $\Av\in \mathcal A_\tv'$ and $t>0$, and let $V\subset \R^d$ be a proper subspace. Fix $n \in \N$ and let $\iii=(i_1,i_2,\ldots),\jjj=(j_1,j_2,\ldots)\in \Sigma$ be such that $\iii \ne \jjj$ and $\bi\wedge \bj\in \Sigma_{n-1}$. Define
  \begin{equation*}
    E_{\iii,\jjj}(G) = v_{i_n} - v_{j_n} + \sum_{k=n+1}^\infty G^{T}A_{\sigma^n(\iii|_k)}G v_{i_{k+1}} - \sum_{k=n+1}^\infty G^{T}A_{\sigma^n(\jjj|_k)}G v_{j_{k+1}}
  \end{equation*}
  and write $w_{\iii,\jjj}(G) = E_{\iii,\jjj}(G)/|E_{\iii,\jjj}(G)|$ for all $G\in O(d)$. Let $D_{\iii \wedge \jjj} = A_{\iii \wedge \jjj}^T \proj_V^T \proj_V A_{\iii \wedge \jjj}$. The claim is equivalent to
  \begin{equation}\label{eq:equivalent_statement}
    \Theta(\{ G \in O(d) : |\la GE_{\iii,\jjj}(G), D_{\iii\wedge\jjj}GE_{\iii,\jjj}(G) \ra| < t^{2} \}) \le C\prod_{i=1}^{\dim V} \min\biggl\{ 1,\frac{t}{\alpha_i(\proj_VA_{\iii\wedge\jjj})} \biggr\}
  \end{equation}
  and therefore, our task is to show that \eqref{eq:equivalent_statement} holds.

  Defining
  \begin{equation*}
    R = \frac{\|\tv\|}{1-\|\Av\|} \quad \text{and} \quad r = |v_{i_n}-v_{j_n}|-(\|A_{i_n}\|+\|A_{j_n}\|)R
  \end{equation*}
  we have
  \begin{equation} \label{eq:E-estimate}
    0<r\le |E_{\bi, \bj}(G)| \le |v_{i_n}-v_{j_n}| + (\|A_{i_n}\|+\|A_{j_n}\|)R \le 2R.
  \end{equation}
  Observe that we also have
  \begin{equation} \label{claim2}
  \begin{split}
    |w_{\iii,\jjj}(G)-w_{\iii,\jjj}(H)| &\le \biggl|E_{\iii,\jjj}(G)\biggl(\frac{1}{|E_{\iii,\jjj}(G)|}-\frac{1}{|E_{\iii,\jjj}(H)|}\biggr) + \frac{E_{\iii,\jjj}(G)-E_{\iii,\jjj}(H)}{|E_{\iii,\jjj}(H)|} \biggr| \\
        &\le r^{-1}2|E_{\iii,\jjj}(G)-E_{\iii,\jjj}(H)| \\
    &\le 4r^{-1}R(\|A_{i_n}\|+\|A_{j_n}\|)\|G-H\|
  \end{split}
  \end{equation}
  for all $G,H \in O(d)$. Furthermore, write
  \begin{equation*}
    c_1 = 4r^{-1}R(\|A_{i_n}\|+\|A_{j_n}\|)
  \end{equation*}
  and note that, by the definition of $\mathcal{A}_{\tv}'$, we have $c_1<1$. Recalling that $O(d)$ is compact, let $\{ B(G_k,2^{-1}t^{2}r^{-2}c_1^{-1}) \}_{k=1}^K$ be a finite maximal packing of $O(d)$ and let $B_k = B(G_k,t^{2}r^{-2}c_1^{-1})$. Notice that $\{B_k\}_{k=1}^K$ is a cover of $O(d)$. Fix $e \in S^{d-1}$ such that $\la e,v_{i_n}-v_{j_n} \ra = 0$ and write $w_k = w_{\iii,\jjj}(G_k)$. Let $H_k \in SO(d)$ be such that it is a rotation in the plane spanned by $w_k$ and $e$ and satisfies
  \begin{equation*}
    w_k = H_ke.
  \end{equation*}

  \begin{claim}\label{claim1}
  It holds that
  \begin{align*}
    \Theta(\{G\in O(d) : \;&|\la GE_{\bi, \bj}(G), D_{\bi\wedge\bj}GE_{\bi, \bj}(G)\ra|<t^2\}) \\
    &\le \sum_{k=1}^K \Theta (\{G\in B_kH_k : |\la Ge, D_{\bi\wedge \bj}Ge\ra| < 2t^2/r^2 \}).
  \end{align*}
  \end{claim}

  \begin{proof}
  Note that without loss of generality, we may assume that $\|D_{\iii\wedge\jjj}\| \le 1$. Therefore,
  \begin{equation}\label{eq:estimate1}
    |\la Gx,D_{\iii\wedge\jjj}Gx \ra - \la Gy,D_{\iii\wedge\jjj}Gy \ra| \le 2|x-y|.
  \end{equation}
  for every $x,y \in S^{d-1}$. Observe that if $G \in B_k$, then \eqref{claim2} implies $|w_{\iii,\jjj}(G)-w_k| < t^{2}/r^2$. Thus, by \eqref{eq:estimate1} and the invariance of $\Theta$ under the action of $SO(d)$, we obtain
  \begin{equation} \label{eq:estimate3}
  \begin{split}
    \Theta(\{ G \in O(d) : \;&|\la GE_{\iii,\jjj}(G), D_{\iii\wedge\jjj}GE_{\iii,\jjj}(G) \ra| < t^{2} \}) \\
    &\le \Theta(\{ G \in O(d) : |\la Gw_{\iii,\jjj}(G), D_{\iii\wedge\jjj}Gw_{\iii,\jjj}(G) \ra| < t^{2}/r^2 \}) \\
    &\le \sum_{k=1}^K \Theta(\{ G \in B_k : |\la Gw_{\iii,\jjj}(G), D_{\iii\wedge\jjj}Gw_{\iii,\jjj}(G) \ra| < t^{2}/r^2 \}) \\
    &\le \sum_{k=1}^K \Theta(\{ G \in B_k : |\la Gw_k, D_{\iii\wedge\jjj}Gw_k \ra| < 2t^{2}/r^2 \})\\
    &\le \sum_{k=1}^K \Theta(\{ G \in B_k : |\la GH_ke, D_{\iii\wedge\jjj}GH_ke \ra| < 2t^{2}/r^2 \}) \\
    &= \sum_{k=1}^K \Theta(\{ G \in B_kH_k : |\la Ge, D_{\iii\wedge\jjj}Ge \ra| < 2t^{2}/r^2 \})
  \end{split}
  \end{equation}
  which is what we wanted.
  \end{proof}

  \begin{claim}\label{claim3}
    There exists a constant $0<\roo<1$ such that
    \begin{equation*}
      \|H_k-H_h\| \le \roo\|G_k-G_h\|.
    \end{equation*}
  \end{claim}

  \begin{proof}
  If $w_k=w_h$, then also $H_h=H_k$ and there is nothing to prove. Furthermore, if $w_k=e$, then
  \begin{equation*}
    \|H_k - H_h\|=\|\Id - H_h\|=2|\sin(\beta/2)|,
  \end{equation*}
  where $\beta$ is the angle by which $H_h$ rotates. Since $H_h(e)=w_h$ and $2\sin(\beta/2)=|w_k - w_h|$, the claim follows from \eqref{claim2}.
  Since the case $w_h=e$ is similar, we may assume that $w_k, w_h$, and $e$ are all distinct. Define
  \begin{equation*}
    W = \linspan\{ w_k,w_h,e \}.
  \end{equation*}
  We have $H_ku = H_hu = u$ for all $u \in W^\perp$ since $H_k$ and $H_h$ are rotations on $W$. Observe that
  \begin{equation} \label{eq:H-estimate1}
  \begin{split}
    \|H_k-H_h\| &= \|H_k^T-H_h^T\| = \sup\{ |(H_k^T-H_h^T)w| : w \in W \text{ such that } |w|=1 \} \\
    &= \sup\{ |(2-2\la H_k^Tw,H_h^Tw \ra)^{1/2}| : w \in W \text{ such that } |w|=1 \} \\
    &= (|2-2\inf\{ \la w,H_kH_h^Tw \ra| : w \in W \text{ such that } |w|=1 \})^{1/2} \\
    &= (2-2\cos\beta)^{1/2} = 2 |\sin(\beta/2)|,
  \end{split}
  \end{equation}
  where $\beta$ is the angle by which $H_kH_h^T$ rotates. Observe that, by \eqref{eq:E-estimate},
  \begin{equation*}
    \left |\biggl\la \frac{v_{i_n}-v_{j_n}}{|v_{i_n}-v_{j_n}|},w_{\iii,\jjj}(G) \biggr\ra \right| \ge \frac{|v_{i_n}-v_{j_n}|-(\|A_{i_n}\|+\|A_{j_n}\|)R}{|E_{\bi, \bj}(G)|} \ge \frac{r}{|v_{i_n}-v_{j_n}|+(\|A_{i_n}\|+\|A_{j_n}\|)R}.
  \end{equation*}
  Therefore, for every $u \in S^{d-1}$ with $u \perp v_{i_n}-v_{j_n}$ we have
  \begin{equation} \label{eq:H-estimate2}
    |u-w_{\iii,\jjj}(G)| \ge \frac{r}{|v_{i_n}-v_{j_n}|+(\|A_{i_n}\|+\|A_{j_n}\|)R}.
  \end{equation}
 Notice that
  \begin{equation*}
    H_k^Tw_k = H_h^Tw_h = e, \quad
    H_kH_h^Tw_h = w_k,
  \end{equation*}
and further,
\begin{equation*}
H_k^Te=2\la e, w_k\ra - w_k,
 \text{ and }
H_h^Te=2\la e, w_h\ra - w_h,
\text{ so that }
    \la e,H_kH_h^Te \ra = \la w_k,w_h \ra.
\end{equation*}
  We see that, under the rotation $H_kH_h^T$, all of the vectors $e,w_k,w_h$ are rotated by the same angle. Furthermore, $w_k$ and $w_h$ are mapped to each other. This is only possible if all the points are within the same distance from the equator of the unit sphere of $\linspan\{w_k,w_h,e\}$; see Figure \ref{fig:sphere}. Therefore the plane of rotation for $H_kH_h^T$ is either $\linspan\{ w_k-e, w_h-e \}$ or $\linspan\{ w_k+e, w_h+e \}$, depending on whether $e$ is on the same hemisphere of the unit sphere with $w_k$ and $w_h$ or not. Let $f \in \{ e,-e \}$ be such that $\linspan\{ w_k-f,w_h-f \}$ is the plane of rotation.

  \begin{figure}
  \subfigure[]{\def\svgwidth{0.43\columnwidth}
        \label{fig:sphere}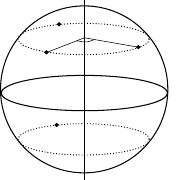}
  \subfigure[]{\def\svgwidth{0.43\columnwidth}
  \label{fig:triangle}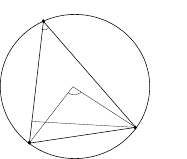}
  \caption{Figure (a) displays the possible relative positions of the points $w_h, w_k$ and $e$. The number $\beta$ is the angle by which $H_kH_h^T$ rotates. Figure (b) represents the triangle with vertices $w_h, w_k$ and $f$ in a plane invariant under $H_kH_h^T$.}
  \end{figure}

  In the plane of rotation, consider the triangle having the points $w_k, w_h$, and $f$ as its vertices, and is of height $m$ (from the vertex $w_k$); see Figure \ref{fig:triangle}. The area $\mathfrak{A}$ of this triangle satisfies
  \begin{equation*}
    \frac{\sin(\beta/2)|w_k-f||w_h-f|}{2} = \mathfrak{A} = \frac{m|w_h-f|}{2} \le \frac{|w_k-w_h||w_k-f|}{2}.
  \end{equation*}
  Therefore, we have
  \begin{equation} \label{eq:H-estimate4}
    \sin(\beta/2) \le \frac{|w_k-w_h|}{|w_k-f|}.
  \end{equation}
  Thus, by \eqref{eq:H-estimate1}, \eqref{eq:H-estimate2}, \eqref{claim2}, and \eqref{eq:H-estimate4}, we get
  \begin{equation*}
    \|H_k-H_h\| \le \frac{4R(\|A_{i_n}\|+\|A_{j_n}\|)(|v_{i_n}-v_{j_n}|+(\|A_{i_n}\|+\|A_{j_n}\|)R)}{r^2}\|G_k-G_h\|.
  \end{equation*}
  Since, by the definition of $\mathcal A_\tv'$,
  $$
    \frac{4R(\|A_{i_n}\|+\|A_{j_n}\|)(|v_{i_n}-v_{j_n}|+(\|A_{i_n}\|+\|A_{j_n}\|)R)}{r^2}<1
  $$
  we have finished the proof of Claim \ref{claim3}.
  \end{proof}

  We will now finish the proof of Proposition~\ref{thm:higher-dim-trans} by verifying  \eqref{eq:equivalent_statement} as a consequence of Claims \ref{claim1} and \ref{claim3}. By Claim~\ref{claim3} and the fact that $\{ B(G_k,2^{-1}t^{2}r^{-2}c_1^{-1}) \}_{k=1}^K$ is a maximal packing of $O(d)$, we have
  \begin{equation*}
    \| G_kH_k - G_hH_h \| \ge \| G_k-G_h \| - \| H_k-H_h \| \ge (1-\roo)2^{-1}t^{2}r^{-2}c_1^{-1}
  \end{equation*}
  for $k \ne h$.
  Since the metric $d(G,H)=\|G-H\|$ is invariant with respect to the group action, the Haar measure $\Theta$ is equivalent to the $d(d-1)/2$-dimensional Hausdorff measure on $O(d)$. Thus, there exists a constant $C$ independent of $t$ such that any $H\in O(d)$ is contained in at most $C$ balls from the family $\{B(G_kH_k,t^{2}r^{-2}c_1^{-1})\}_{k=1}^K$. Hence, by Claim \ref{claim1},
  \begin{align*} 
    \Theta(\{G\in O(d) : \;&|\la GE_{\bi, \bj}(G), D_{\bi\wedge\bj}E_{\bi, \bj}(G)\ra|<t^2\}) \\
    &\le \sum_{k=1}^K \Theta(\{ G \in B_kH_k : |\la Ge, D_{\iii\wedge\jjj}Ge \ra| < 2t^{2}/r^2 \})\\
    &\le C\Theta(\{ G \in O(d) : |\la Ge, D_{\iii\wedge\jjj}Ge \ra| < 2t^{2}/r^2 \}) \\
    &= C\sigma^{d-1}(\{x \in S^{d-1} : |\la x, D_{\iii\wedge\jjj}x \ra| < 2t^{2}/r^2\}),
  \end{align*}
  where $\sigma^{d-1}$ is the spherical measure on the unit sphere $S^{d-1}$. Since $\dim V\leq d-1$, we have $\alpha_d(\proj_V A_{\iii \wedge \jjj})=0$, and
  \begin{align*}
    C\sigma^{d-1}(\{x \in S^{d-1} : \;&|\la x, D_{\iii\wedge\jjj}x \ra| < 2t^{2}/r^2\})\\
    &\le C\mathcal{L}^{d-1}(\{(x_1,\dots,x_{d-1}) \in B(0,1) : \sum_{l=1}^{d-1}\alpha_l(\proj_V A_{\iii \wedge \jjj})^2x_l^2 < 2t^{2}/r^2\})\\
    &\le C\min\{1,\sqrt{2}/r\}^{d-1}\prod_{l=1}^{\dim V}\min\biggl\{1,\frac{t}{\alpha_l(\proj_V A_{\iii \wedge \jjj})}\biggr\}.
  \end{align*}
  This proves \eqref{eq:equivalent_statement} and finishes the proof.
\end{proof}

Our next lemma shows that almost every matrix tuples have simple Lyapunov spectra with respect to any Bernoulli measure. The proof of the lemma is similar to that of \cite[Theorem~7.12]{Vianabook}.
We say that $\mathbf{A}=(A_1,\dots,A_N) \in GL_d(\R)^N$ is \emph{pinching} if there exists $\ii\in\Sigma_*$ such that $A_{\ii}$ has only real eigenvalues with distinct absolute value. Furthermore, we say that $\mathbf{A}$ is \emph{twisting} if there exists $\jj\in\Sigma_*$ such that $A_{\jj}E\cap F=\{0\}$ for all invariant subspaces $E$ and $F$ of $A_{\ii}$ with dimension $\dim E+\dim F\leq d$. Assuming $\mathbf{A}$ to be pinching and twisting guarantees the Lyapunov exponents to be distinct; see \cite[Theorem A]{Avila_Viana}. In the study of self-affine sets, this is often a useful property; for example, see \cite[Theorem 3.1 and Remark 3.2]{KaenmakiKoivusaloRossi2015}.

\begin{lemma}\label{lem:fullmeasureA}
  Let $\mathcal{M}$ be the product of Haar measures on $GL_d(\R)^N$, i.e.\
  \begin{equation*}
    \mathcal{M}(\mathcal{B})
    = \int_{\mathcal{B}} \prod_{i=1}^N \biggl( \frac{1}{\det A_i} \biggr)^d \dd\LL^{d^2N}(\Av)
  \end{equation*}
  for all measurable $\mathcal{B} \subset GL_d(\R)^N$.
  Then
  $$
  \mathcal{M}(\{ \Av \in GL_d(\R)^N : \text{there exist $\pv\in(0,1)^N$ and $i\neq j$ such that }\chi_i(\Av,\pv)=\chi_j(\Av,\pv) \})=0.
  $$
\end{lemma}

\begin{proof}
By \cite[Theorem~A]{Avila_Viana}, it is sufficient to show that the set of $\Av$ without pinching and twisting condition is contained in countably many $d^2N-1$ dimensional manifolds. It is easy to see that without loss of generality, we may assume that $N=2$.

We start by sketching the argument in the case $d=2$. This follows from \cite[Theorem A]{Avila_Viana}, where the assumptions of the theorem are satisfied by the following argument: After perturbing one of the maps by a small rotation, we see that the twisting condition holds. The pinching condition, in the case where $A_1$ or $A_2$ have only real eigenvalues, follows immediately by slightly perturbing the matrix having real eigenvalues. Suppose then that both $A_1$ and $A_2$ have complex eigenvalues. Then we may consit $A_1$ to be the
composition of a rotation with a dilation. Up to perturbation, we may
suppose that the rotation is irrational. Up to another perturbation, we
may suppose that $A_2$ is not conformal: there is a cone $C$ whose image
$A_2(C)$ is strictly slimmer. Then we may find arbitrarily large values of $n$
such that $B=A_1^n A_2 (C)$ is contained in $C$.
This implies that the eigenvalues of $B$ are real and different in absolute value and hence, the pinching condition holds.

Henceforth we assume that $d\geq3$. Curiously, this assumption is needed in our argument. Let $\lambda_1(A),\ldots,\lambda_d(A)$ be the eigenvalues of a matrix $A$ written in a decreasing order by the absolute value, i.e.\ $|\lambda_i(A)|\geq|\lambda_{i+1}(A)|$ for all $i$. Let $A_1,A_2\in GL_d(\R)$. If $\lambda_i(A_j)\in\R$, then $|\lambda_i(A_j)|\neq|\lambda_k(A_j)|$ for $k\neq i$. If $\lambda_i(A_j)=\overline{\lambda_{i+1}(A_j)}\in\mathbb{C}$, then $|\lambda_i(A_j)|\neq|\lambda_k(A_j)|$ for $k\notin \{i,i+1\}$. Observe that in this case, by an arbitrary small pertubation on the matrices $A_1$ and $A_2$, we may assume that the argument of the eigenvalue $\lambda_i(A_j)$ is irrational. Similarly, we may also assume that the eigenvalues are rationally independent, i.e.\ none of the eigenvalues is not rational multiple of each other.

  Let us assume that $A_j$ has $r_j$ real and $c_j$ complex eigenvalues. Then $r_j+2s_j=d$ and, moreover, all eigenspaces of $A_j$ have dimension either $1$ or $2$. Let us denote the eigenspaces of $A_j$ by $\xi_{e(1)}(A_j),\dots,\xi_{e(d)}(A_j)$ corresponding to the eigenvalues $\lambda_1(A_j),\dots,\lambda_d(A_j)$. This means that if $\lambda_i(A_j)\in\R$, then $\dim\xi_{e(i)}(A_2)=1$ and $\xi_{e(i)}(A_2)$ is the eigenspace of $\lambda_i(A_j)\in\R$. On the other hand, if $\lambda_i(A_j)=\overline{\lambda_i(A_j)}\in\mathbb{C}$, then $\xi_{e(i)}(A_2)=\xi_{e(i+1)}(A_2)$, $\dim\xi_{e(i)}(A_2)=2$, and $\xi_{e(i)}(A_2)$ is the eigenspace of $\lambda_i(A_j)$. By an arbitrary small perturbation, we may assume that the eigenspaces are in general position: for every $i \in \{1,\ldots,d\}$ there exist non-zero subspaces $E_1,E_2\subset\xi_{e(i)}(A_1)$ and $F_1,F_2\subset\xi_{e(i+1)}(A_2)$ such that
  \begin{equation}\label{eq:generic}
  \begin{split}
    \mathrm{span}\{\xi_{e(1)}(A_1),\dots,E_1\}\cap\mathrm{span}\{F_1,\dots,\xi_{e(d)}(A_2)\}=\{0\},\\
    \mathrm{span}\{\xi_{e(1)}(A_2),\dots,E_2\}\cap\mathrm{span}\{F_2,\dots,\xi_{e(d)}(A_1)\}=\{0\}.
  \end{split}
  \end{equation}
  We define a flag of cones $C_1\times C_2\times\cdots\times C_{d-1}$ in $G(1,d)\times G(2,d)\times\cdots\times G(d-1,d)$ by taking sufficiently small neighbourhoods of the flag of invariant subspaces $\mathrm{span}\{\xi_{e(1)}(A_1),\dots,\xi_{e(i)}(A_1)\}$ and sufficiently small neighbourhoods of an $i$-dimensional subspace of $\mathrm{span}\{\xi_{e(1)},\dots,\xi_{e(i)}(A_1)\}$, which contains $\mathrm{span}\{\xi_{e(1)},\dots,\xi_{e(i-1)}(A_1)\}$ if $\lambda_i(A_1)=\overline{\lambda_{i+1}(A_1)}\in\mathbb{C}$ such that it satisfies the following six conditions:
  \begin{enumerate}
    \item every $C_i$ is open and simply connected in $G(i,d)$,
    \item for every $E\in C_i$ there exists $F\in C_{i+1}$ such that $E\subset F$,
    \item if $\lambda_i(A_1)=\overline{\lambda_{i+1}(A_1)}\in\mathbb{C}$, then $C_i$ contains  $\mathrm{span}\{\xi_{e(1)}(A_1),\dots,\xi_{e(i-1)}(A_1),E\}$, where $E$ is as in \eqref{eq:generic}, and $\overline{C_i}$ is transversal to $\mathrm{span}\{E',\xi_{e(i+2)}(A_1),\dots,\xi_{e(d)}(A_1)\}$ for a subspace $E'$ with $\mathrm{span}\{E,E'\}=\xi_{e(i)}(A_1)$,
    \item if $\lambda_i(A_1)\in\R$, then $C_{i}$ contains $\mathrm{span}\{\xi_{e(1)}(A_1),\dots,\xi_{e(i)}(A_1)\}$, and $\overline{C_i}$ is transversal to $\mathrm{span}\{\xi_{e(i+1)}(A_1),\dots,\xi_{e(d)}(A_1)\}$,
    \item if $\lambda_i(A_2)\in\R$, then $\mathrm{span}\{\xi_{e(1)}(A_2),\dots,\xi_{e(i)}(A_2)\}$ is not contained in $\overline{C_i}$ and $\overline{C_i}$ is transversal to $\mathrm{span}\{\xi_{e(i+1)}(A_2),\dots,\xi_{e(d)}(A_2)\}$,
    \item if $\lambda_i(A_2)=\overline{\lambda_{i+1}(A_2)}\in\mathbb{C}$, then $\overline{C_i}$ is transversal to $\mathrm{span}\{F,\xi_{e(i+2)}(A_2),\dots,\xi_{e(d)}(A_2)\}$ and $\mathrm{span}\{\xi_{e(1)}(A_2),\dots,F\}\notin\overline{C_{i}}$ for all proper subspaces $F\subset\xi_{e(i)}(A_2)$, $\overline{C_{i+1}}$ is transversal to $\mathrm{span}\{\xi_{e(i+2)}(A_2),\dots,\xi_{e(d)}(A_2)\}$, and $\mathrm{span}\{\xi_{e(1)}(A_2),\dots,\xi_{e(i+1)}(A_2)\}\notin\overline{C_{i}}$.
  \end{enumerate}
  Here two collections of subspaces are transversal if they form a positive angle.

  Observe that, by the properties of the eigenvalues, for both $j\in\{1,2\}$, if $\lambda_i(A_j)\in\R$, then
  \begin{equation*}
    A_j^k(C_i)\to\mathrm{span}\{\xi_1(A_j),\dots,\xi_{e(i)}(A_j)\}
  \end{equation*}
  uniformly. Moreover, if $\lambda_i(A_j)=\overline{\lambda_{i+1}(A_j)}\in\mathbb{C}$, then
  \begin{equation*}
    A_j^k(C_{i+1})\to\mathrm{span}\{\xi_1(A_j),\dots,\xi_{e(i+1)}(A_j)\} \quad \text{and} \quad
    d(A_j^k(C_i),\mathrm{span}\{\xi_1(A_j),\dots,\xi_{e(i)}(A_j))\to 0
  \end{equation*}
  uniformly. Note that in this case, $\mathrm{diam}(A_j^k(C_i))$ does not tend to zero. Thus, if $i$ is an index such that either $\lambda_i(A_1)\in\R$, or $\lambda_i(A_2)\in\R$, or $\lambda_{i-1}(A_1)=\overline{\lambda_{i}(A_1)}\in\mathbb{C}$ and $\lambda_{i-1}(A_2)=\overline{\lambda_{i}(A_2)}\in\mathbb{C}$, then there exist $n_1,n_2 \in \N$ such that
  \begin{equation}\label{eq:conv1}
    A_1^{n_2}A_2^{n_1}(\overline{C_i})\subset C_i.
  \end{equation}
  Let $i_1,\dots,i_p$ be the indices for which $\lambda_{i_j}(A_1)=\overline{\lambda_{i_j+1}(A_1)}\in\mathbb{C}$ and $\lambda_{i_j}(A_2)=\overline{\lambda_{i_j+1}(A_2)}\in\mathbb{C}$. Let $v_1^j,\dots,v_d^j$ be a basis of $\R^d$ such that $v_i^j\in\xi_{e(i)}(A_j)$. Observe that there exists a subsequence $(m_q)_{q \in \N}$ such that $(\lambda_1(A_k)\cdots\lambda_{i_j-1}(A_k)|\lambda_{i_j}(A_k)|)^{-m_q} (A_k^{\wedge i_j})^{m_q}\to P_k$, where $P_k$ is the projection onto $V_k=\mathrm{span}\{v^k_1\wedge\cdots\wedge v_{i_j-1}^k\wedge v_{i_j}^k,v^k_1\wedge\cdots\wedge v_{i_j-1}^k\wedge v_{i_j+1}^k\}$ along the invariant subspace $W_k$ of $A_k^{\wedge i_j}$ transversal to $V_k$ so that $\dim V_k+\dim W_k=\binom{d}{i_j}$. Note that if $v\in\wedge^{i_j}\R^d$, then $\|P_kv\|\leq\|v\|$, and if $v\notin V_k$, then $\|P_kv\|<\|v\|$. Since by a sufficiently small perturbation, we may assume that the subspaces $V_1, W_1$ and $V_2, W_2$ are in general position, there exists a unique $a_k\in V_k$ with $\|a_k\|=1$ such that $\sup_{v\in V_{k}}\|P_{3-k}v\|/\|v\|=\|P_{3-k}a_{k}\|$. Thus, there exist cones $C_k'$ in $V_k$ with arbitrary large diameter such that $\mathrm{diam}(P_{3-k}(C'_k))< c'\diam{C'_k}$ and $a_k\in C'_k$ for some $0<c'<1$. Let us denote the rotation on $V_k$, which maps $P_{k}a_{3-k}$ to $a_k$, by $O_k$. Therefore, $O_kP_{k}O_{3-k}P_{3-k}(\overline{C'_k})\subset C'_k$. By the irrationality of the argument of the eigenvalues, one can choose $q$ sufficiently large and $n_1, n_2 \in \N$ such that $(A_1^{\wedge i_j})^{n_1}(A_1^{\wedge i_j})^{m_q}(A_{2}^{\wedge i_j})^{n_2}(A_{2}^{\wedge i_j})^{m_q}(\overline{C'_k})\subset C'_k$. Using the natural correspondence between $G(d,i_j)$ and $\wedge^{i_j}\R^d$, the set $C_k'$ corresponds to a cone $C_k''$ in $\mathrm{span}\{\xi_{e(1)}(A_k),\dots,\xi_{e(i_j)}(A_k)\}$ and hence, by choosing the cone $C_{i_j}$ in $G(d,i_j)$ sufficiently close to $C_1''$, we get
  \begin{equation}\label{eq:conv3}
    A_1^{n_1}A_1^{m_q}A_{2}^{n_2}A_{2}^{m_q}(\overline{C_{i_j}})\subset C_{i_j}
  \end{equation}
  for all $j\in\{1,\dots,p\}$.
  Thus, by irrationality and rational independence of the arguments of the complex eigenvalues of $A_1$ and $A_2$, \eqref{eq:conv1} and \eqref{eq:conv3} imply that there exist $n_1',n_2'\in \N$ such that
  $$
    A_1^{n_2'}A_2^{n_1'}(\overline{C_i})\subset C_i
  $$
  for all $i\in\{1,\dots,d\}$.
  Therefore, $A_\iii = A_1^{n_2'}A_2^{n_1'}$ has only real eigenvalues with distinct absolute value. This shows that $(A_1,A_2)$ is pinching. By a similar argument, there exist $m_1',m_2' \in \N$ such that $A_\jjj=A_2^{m_2'}A_1^{m_1'}$ has only real eigenvalues with distinct absolute value and $E\cap F=\{0\}$ for all invariant subspaces $E$ of $A_{\ii}$ and $F$ of $A_{\jj}$ with $\dim E+\dim F\leq d$. This proves the twisting property.
\end{proof}

Recall that $\mathbb A$ is the collection of all contractive matrix tuples having distinct Lyapunov exponents with respect to an ergodic equilibrium state, see \S \ref{sec:sub-systems}, and $\mathcal{D}$ is defined in \eqref{eq:d}.

\begin{proof}[Proof of Theorem~\ref{thm:main}]
Let $\mathcal{M}$ be as in Lemma \ref{lem:fullmeasureA}. Since $\mathcal{L}^{d^2N}\ll\mathcal{M}$ it suffices to check that the assumptions of Theorem~\ref{thm:aeformeasure} hold for the measure $\mathcal{M}$.
By the Haar property, for any measurable $L^1(\mathcal{M})$-function $f\colon GL_d(\R)^N\to\R$ and for every $G\in O(d)$, we have
\begin{align*}
  \int f(\Av)\,\mathrm{d}\mathcal{M}(\Av) & = \idotsint f(A_1,\dots,A_N)\prod_{i=1}^N\biggl(\frac{1}{\det A_i}\biggr)^d\,\mathrm{d}\mathcal{L}^{d^2}(A_1)\cdots\mathrm{d}\mathcal{L}^{d^2}(A_N) \\
  &= \idotsint f(G^T A_1G,\dots,G^TA_NG)\prod_{i=1}^N\biggl(\frac{1}{\det A_i}\biggr)^d\,\mathrm{d}\mathcal{L}^{d^2}(A_1)\cdots\mathrm{d}\mathcal{L}^{d^2}(A_N).
\end{align*}
Thus, $\mathcal{M}=\int H_{\Av}\Theta\,\mathrm{d}\mathcal{M}(\Av)$, where $H_{\Av}(G)=(G^TA_1G,\dots,G^TA_NG)$.

By Lemma~\ref{lem:fullmeasureA}, for $\mathcal{M}$-almost every $\Av$, all the Bernoulli measures have simple Lyapunov spectra and thus, by applying Lemma~\ref{lem:safreg}, any Bernoulli measure satisfies the Ledrappier-Young formula. Hence, the first statement follows by the combination of Lemma~\ref{lem:ssc-d}, Proposition~\ref{thm:higher-dim-trans}, and Theorem~\ref{thm:aeformeasure}.

To turn to the set dimension statement, observe that by \eqref{eq:domsplit}, for every $\Av\in\mathcal{D}$, the $s$-equilibrium state $\mu_{\Av}$ has simple Lyapunov spectrum for $s = \dimaff\Av$. Moreover, by Lemma~\ref{lem:quasi} the $\mu_{\Av}$ is quasi-Bernoulli. Thus, by Lemma~\ref{lem:safreg}, $\mu_{\Av}$ satisfies the Ledrappier-Young formula. Therefore, by the combination of Lemma~\ref{lem:ssc-d}, Proposition~\ref{thm:higher-dim-trans}, Theorem~\ref{thm:almosteverymatrix}, and the fact that
$$
  \diml\mu_{\Av}=\dimaff\Av,
$$
the second assertion follows.
\end{proof}

\section{Further discussion and questions}

We finish the article by posing couple of questions. An affirmative answer for either of the two following questions would immediately improve Theorem \ref{thm:main}.

\begin{question}
  Recall that $\mathbb{A}$ is the collection of all tuples $\Av \in GL_d(\R)^N$ of contractive matrices that satisfy $\chi_i(\Av,\mu) \ne \chi_j(\Av,\mu)$ for $i \ne j$ where $\mu$ is an ergodic $s$-equilibrium state of $\Av$ and $s=\dimaff(\Av)$. Since $\mathcal{D}$ defined in \eqref{eq:d} is open, the set $\mathbb{A}$ contains interior points; see \S \ref{sec:higherdimcase}. Is $\mathbb{A}$ an open and dense subset of $GL_d(\R)^N$ with full Lebesgue measure?
\end{question}

In the planar case, the above question is already addressed in \cite[Theorem 3.1]{MorrisShmerkin2016}; see also \cite[Theorem 13]{Morris2016}.

\begin{question}
  Can every $s$-equilibrium state of $\Av$ for $s=\dimaff\Av$ be approximated by step-$n$ Bernoulli measures? More precisely, does there, for every $\eps>0$, exist $n \in \N$ and a step-$n$ Bernoulli measure $\mu$ such that $\diml\mu \ge \dimaff\Av - \varepsilon$? Observe that, by \cite[\S 3.2]{MorrisShmerkin2016}, this cannot be done with fully supported step-$n$ Bernoulli measures.
\end{question}

It would also improve our results if the $s$-equilibrium state of $\Av$ for $s=\dimaff\Av$ turned out to be quasi-Bernoulli for $\LL^{d^2N}$-almost all $\Av \in GL_d(\R)^N$. Recall that, by \cite[Propositions 3.4 and 3.6]{KaenmakiBing}, the $s$-equilibrium state is unique and satisfies a certain Gibbs property in a set of full Lebesgue measure. However, the proposition below implies that the quasi-Bernolli property does not hold generically.

Let us for simplicity assume that $d=2$. We say that a matrix $A$ is \emph{hyperbolic} if it has real unequal eigenvalues, \emph{elliptic} if it has non-real eigenvalues, and \emph{irrational elliptic} if it has non-real eigenvalues whose arguments are irrational multiples of $\pi$.

\begin{proposition} \label{thm:morris}
  Suppose that $\Av = (A_1,\ldots,A_N) \in GL_2(\R)^N$ is irreducible and generates a semigroup which contains a hyperbolic matrix and an irrational elliptic matrix. Then for every $0<s<2$ the $\fii^s$-equilibrium state of $\Av$ is not quasi-Bernoulli.
\end{proposition}

\begin{proof}
  Let $\mathcal{S}$ be the semigroup generated by $\Av$. Let $0<s<2$ and $\mu$ be the $\fii^s$-equilibrium state. Suppose for a contradiction that $\mu$ is quasi-Bernoulli. Let $X_1 \in \mathcal{S}$ be a hyperbolic matrix and $X_2 \in \mathcal{S}$ be an irrational elliptic matrix. By changing the basis, we may, without loss of generality, assume that $X_2$ is a scalar multiple of a rotation matrix. Since $\mu$ is quasi-Bernoulli and, by irreducibility, satisfies the Gibbs property, there exists $K \ge 1$ such that
  \begin{equation} \label{eq:estimate}
    K^{-1}\|B_1\|\|B_2\| \le \|B_1B_2\| \le \|B_1\|\|B_2\|.
  \end{equation}
  Let $\mathcal{S}' = \overline{\mathbb{R}\mathcal{S}}$ be the closure of the smallest homogeneous semigroup containing $\mathcal{S}$. Note that \eqref{eq:estimate} holds for all $B_1,B_2 \in \mathcal{S}'$. The closure of $\{ \det(X_2)^{n/2}X_2^n : n \in \N \}$ is $SO(2)$ and is contained in $\mathcal{S}'$. Also, $P=\lim_{n \to \infty} X_1^n/\|X_1^n\|$ is rank one projection and is contained in $\mathcal{S}'$. Choose $R \in SO(2) \subset \mathcal{S}'$ such that $R(\mathrm{img}(P)) \subset \mathrm{ker}(P)$. Thus $PRP=0$, $\|P\| = \|RP\| = 1$, and $P,R \in \mathcal{S}'$. Therefore, $0<K^{-1} \le \|PRP\| = 0$ which is a contradiction.
\end{proof}

Since $\mathcal{A} = \{ \Av : \text{the semigroup generated by } \Av \text{ contains an elliptic matrix} \}$ has positive Lebesgue measure and, after a small perturbation, any $\Av \in \mathcal{A}$ generates a semigroup that contains an irrational elliptic matrix and a hyperbolic matrix, we see that the $s$-equilibrium state of $\Av$ for $s=\dimaff\Av$ is not quasi-Bernoulli for $\LL^{4N}$-almost all $\Av \in GL_2(\R)^N$.

We finish the paper by posing the following question.

\begin{question}
  Does it hold that, for Lebesgue-almost every $\Av = (A_1,\ldots,A_N) \in GL_d(\mathbb{R})^N$, either $\Av \in \mathcal{D}$ or for every $0<s<d$ the $\varphi^s$-equilibrium state of $\Av$ is not quasi-Bernoulli.
\end{question}

This question is related to the question of Yoccoz \cite[Question 4]{Yoccoz2004} (see also \cite[Question 4]{AvilaBochiYoccoz2010}).

\begin{ack}
  The authors thank Marcelo Viana for helpful discussions related to Lyapunov exponents, Ian Morris for pointing out Proposition \ref{thm:morris}, and the anonymous referee for suggesting Proposition \ref{prop:report}. 
\end{ack}

\bibliography{vaitbib}
\bibliographystyle{abbrv}
\end{document}